\documentclass[a4paper]{article}
\usepackage[english]{babel}
\usepackage[utf8]{inputenc}
\usepackage[T1]{fontenc}

\usepackage{lipsum}
\usepackage{amssymb, amsthm, amsmath}
\usepackage{enumitem}
\usepackage{bbm,dsfont}

\newcommand{\R}{\mathbb{R}}
\newcommand{\C}{\mathbb{C}}
\newcommand{\N}{\mathbb{N}}
\newcommand{\A}{\mathcal{A}}

\renewcommand{\L}{\mathcal{L}}

\newcommand{\B}{\mathcal{B}}

\renewcommand{\Re}{\operatorname{Re}}

\newcommand{\fra}{\mathfrak{a}}

\newcommand{\frA}{\mathfrak{A}}
\newcommand{\frB}{\mathfrak{B}}
\renewcommand{\mid}{\, \vert \,}
\newcommand{\MR}{\textit{MR}\,}

\theoremstyle{plain}
\newtheorem{theorem}{Theorem}[section]
\newtheorem{proposition}[theorem]{Proposition}
\newtheorem{lemma}[theorem]{Lemma}
\newtheorem{corollary}[theorem]{Corollary}
\theoremstyle{definition}
\newtheorem{definition}[theorem]{Definition}
\newtheorem{problem}[theorem]{Problem}

\newtheorem{remark}[theorem]{Remark}

\begin{document}
\title{ Uniform approximation of non-autonomous evolution equations \footnote {}}
\author{
O. EL-Mennaoui, Agadir,  H.Laasri, Hagen
}


\maketitle

\begin{abstract}\label{abstract}
We study  $L^2$-maximal regularity in a Hilbert space $H$ for non-autonomous linear evolution equations of the form
\begin{equation}\label{Abstract equation}
  \dot u(t)+\A(t)u(t)=f(t)\ \ t\in[0,T],\ \
 u(0)=u_0.
  \end{equation}
where $\A(t),\ t\in [0,T]$  arise from  a non-autonomous sesquilinear forms $\fra(t,\cdot,\cdot)$  with constant domain $V\subset H.$ $L^2-$maximal regularity result is proved recently in \cite{Ar-Mo15} when  $\fra$ is H\"older continuous of type $\alpha>1/2.$ In this paper we recover the same results by the approximation method developed in \cite{ELLA13}, \cite{LASA14} and \cite{ELLA15}. The method uses an appropriate approximation $\A_\Lambda(\cdot)$ of $\A(\cdot)$ for which
\begin{equation}\label{Abstract equation approx}
  \dot u_{\Lambda}(t)+\A_{\Lambda}(t)u_{\Lambda}(t)=f(t)\ \ t\in[0,T],\ \
 u_{\Lambda}(0)=u_0
  \end{equation}
  has $L^2$-maximal regularity where $\Lambda$ is a subdivision of $[0,T].$
Furthermore, we show that there exists a sequence $(\Lambda_n)_{n\in\N}$ of subdivisions  of $[0,T]$ depending on the modulus of continuity  such that the sequenece of the solutions
$u_{\Lambda_n}$ of (\ref{Abstract equation approx}) converges in $L^2(0,T,V)\cap H^1(0,T,H)\cap C(0,T,V)$ uniformly on the initial datas $(u_0,f)$ to the solution $u$ of (\ref{Abstract equation}) as $n
\rightarrow 0.$ Moreover, we show that such an uniform converges with respect to initial datas holds for arbitrary subdivision of $[0,T]$ under a little more assumptions on the modulus of continuity.

\end{abstract}

\bigskip
\noindent
\textbf{keywords}: Sesquilinear forms, non-autonomous evolution equations, maximal regularity, approximation. \medskip


\section*{Introduction\label{s1}}
Let $V,H$ be two separable Hilbert space such that $V$ is continuously and densely embedded into $H.$  Consider a non-autonomous form
\[
	\fra: [0,T]\times V\times V \mapsto \mathbb C
\]
such that   $\fra(t, .,.)$ is sesquilinear for all $t\in[0,T]$, $\fra(.,u,v)$ is measurable for all $u,v\in V,$
\begin{equation*}\label{eq:continuity-nonaut}
	| \fra(t,u,v) | \le M \|u\|_V \|v\|_V \quad t\in[0,T], u,v\in V,\ (\text{boundedness})
\end{equation*}
and
\begin{equation*}\label{eq:Ellipticity-nonaut}
	\Re~ \fra (t,u,u) \ge \alpha \|u\|^2_V \quad t\in[0,T], v\in V,\ (\text{coerciveness})
\end{equation*}
for some $\alpha>0$ and  $M\geq 0.$ For each $t\in[0,T]$ we associate a unique operator $\A(t)\in \L(V,V')$  such that
\[\fra(t,u,v)=\langle\A(t)u,v\rangle \quad\hbox{ for all } u,v\in V.\]

 \par\noindent Then we say that the non-autonomous Cauchy problem
\begin{equation}\label{Abstract Cauchy problem 0}
\dot{u} (t)+\A(t)u(t)=f(t), \quad u(0)=u_0
\end{equation}
has \textit{$L^2$-maximal regularity in $H$} if for every $f\in L^2(0,T,H)$ and $u_0\in V$ there exists a unique function $u$ belonging  to $\MR(V,H):=L^2(0,T;V)\cap H^1(0,T;H)$ such that $u$ satisfies  (\ref{Abstract Cauchy problem 0}). \par\noindent Considering (\ref{Abstract Cauchy problem 0}) on $V',$  Lions proved on 1961 (see \cite{Lio61} or \cite[p. 620]{DL88}) the following \textit{$L^2$-maximal regularity in $V'$} result:

\begin{theorem}(\textrm{Lions})\label{wellposedness in V'}
 For all $f\in L^2(0,T;V^\prime)$ and $u_0\in H,$ the problem (\ref{Abstract Cauchy problem 0})
has a unique solution $u \in \MR(V,V'):=L^2(0,T;V)\cap H^1(0,T;V').$
\end{theorem}

 Theorem \ref{wellposedness in V'} requires  only the measurability of  $t\mapsto \fra(t,u,v)$  for all $u,v\in V$. However, in applications to  boundary valued problems maximal regularity in $V'$ is not sufficient. Only the part $A(t)$ of $\A(t)$ in $H$ does realize the boundary conditions in question. One is more interested on $L^2$-maximal regularity in $H:$
\begin{problem}\label{Lions's problem}
If $f\in L^2(0,T; H)$ and $u_0\in V$, does the solution $u$ of (\ref{Abstract Cauchy problem 0}) belong to $H^1(0,T;H)$ ?
\end{problem}
This problem is asked by Lions in \cite[p.\ 68]{Lio61} for $u_0=0$ and $\fra(t,u,v)=\overline{\fra(t,v,u)},$ i.e., $\fra$ is symmetric.  A recent result by  Dier \cite{D1}, show that the answer of this question is negative in general.On the other hand, some positive results are due to Lions 
\cite[p.~68, p.~94, ]{Lio61}, \cite[Theorem~1.1, p.~129]{Lio61} and
\cite[Theorem~5.1, p.~138]{Lio61} and to Bardos \cite{Bar71}
under additional regularity assumptions on the form $\fra,$   the
initial value $u_0$ and the inhomogeneity $f.$ More recently, this
problem has been studied  with some progress and different
approaches by  Arendt, Dier, Laasri and Ouhabaz \cite{ADLO14}, Arendt
and Monniaux \cite{Ar-Mo15}, Ouhabaz \cite{O15}, Dier \cite{D2}, Haak and Ouhabaz
\cite{OH14}, Ouhabaz and Spina
\cite{OS10} and Dier and Zacher \cite{Di-Za16}. Results on multiplicative perturbation are established in \cite{ADLO14, D2, AuJaLa14}. 
\medskip

In this paper we are interested with the following nice  result due to Arendt and Monniaux \cite{Ar-Mo15}:
\begin{theorem}\label{thm: Arendt-Monniaux} Assume that  $D(A(0)^{1/2})=V$ and that  there exists  $0\leq \gamma<1$ and a continuous function $\omega:[0,T]\longrightarrow [0,+\infty)$ with
\begin{equation}\label{eq: thm Arendt-Monniaux}\sup_{t\in[0,T]} \frac{\omega(t)}{t^{\gamma/2}}<\infty \quad \text{ and }
\int_0^T\frac{\omega(t)}{t^{1+\gamma/2}}<\infty
\end{equation}
such that
\begin{equation*}
    |\fra(t,u,v)-\fra(s,u,v)| \le\omega(|t-s|) \Vert u\Vert_{V} \Vert v\Vert_{V_\gamma}\quad \ (t,s\in[0,T], u,v\in V)
\end{equation*}
where $V_\gamma:=[H,V]$ is the complex interpolation space.
Then the Cauchy problem $(\ref{Abstract Cauchy problem 0})$ has $L^2$-maximal regularity in $H$. Moreover, for each  $f\in L^2(0,T,H)$ and $u_0\in V$ the solution $u$ of  $(\ref{Abstract Cauchy problem 0})$ is continuous on $[0,T]$ with values in $V.$
\end{theorem}
The aim of this paper is to give an explicit approximation
of the problem $(\ref{Abstract Cauchy problem 0})$ under the assumption of Theorem \ref{thm: Arendt-Monniaux}, which is very useful to obtain qualitative properties of the unknown solution $u$ of $(\ref{Abstract Cauchy problem 0}).$  The method employs an approximation by discretisation  of the function  $\A(.):[0,T]\mapsto \L(V,V')$ and then taking a suitable limit.  Namely, let
$\Lambda:=(0=\lambda_0<\lambda_1<...<\lambda_{n+1}=T)$ be a
subdivision of $[0,T].$ Consider an approximation
$\A_\Lambda:\ [0,T]\rightarrow \mathcal{L}(V,V')$ of $\A$ given
by
\begin{equation*}
 \A_{\Lambda}(t):=\frac{\lambda_{k+1}-t}{\lambda_{k+1}-\lambda_k}\A_{k}+\frac{t-\lambda_k}{\lambda_{k+1}-\lambda_k}\A_{k+1} \quad\hbox{   for  } t\in [\lambda_k,\lambda_{k+1}]
\end{equation*}  with

 \[\displaystyle \A_ku
:=\frac{1}{\lambda_{k+1}-\lambda_k}
\int_{\lambda_k}^{\lambda_{k+1}}\A(r)u{\rm  d}r,\quad  u\in V,
k=0,1,...,n.\]  The integral above makes sense since  $t\mapsto \mathcal A(t)u$ is Bochner integrable on $[0,T]$ with values in $V'$ for all $u\in V.$ Note that $\|\mathcal A(t)u\|_{V'}\leqslant M \|u\|_V$ for all $u\in V$ and all $t\in [0,T]$ and $t\mapsto \mathcal A(t)$ is strongly measurable by the Pettis' Theorem \cite[Theorem 1.1.1]{ABHN11}. This is true  since $V$ and $H$ are separable and $t\mapsto \mathcal A(t)$ is weakly measurable.
\par\noindent We prove that for all $u_0\in V$ and
$f\in L^2(0,T;H),$  the non-autonomous problem
 \begin{equation}\label{nCP in V'0}
\dot u_\Lambda (t)+\A_\Lambda(t)u_\Lambda(t)=f(t), \quad  a.e.\  \text{on}\ (0,T), \quad u_\Lambda(0)=u_0
\end{equation}
 \noindent has a unique solution $u_{\Lambda}\in MR(V,H)\cap C(0,T,V),$ and $(u_\Lambda)_\Lambda$ converges weakly in $MR(V,H)$ as $|\Lambda|\to 0,$ and the weak limit $u:=\lim\limits_{ |\Lambda|\to
0}u_{\Lambda}$ solves uniquely $(\ref{Abstract Cauchy problem 0}).$ This provides an alternative proof of Theorem \ref{thm: Arendt-Monniaux} and an approximation of the solution. Moreover, we show that for each null sequence $(t_n)_{n\in\N}\subset \R_+$ such that
\begin{equation*}\label{condition supp pour convergence uniforme0}
\lim\limits_{n\to \infty}\frac{\omega(t_n)}{t_n^{\gamma/2}}=0
\end{equation*}
and all uniforme subdivision  $(\Lambda_n)_{n\in\N}$ of $[0,T]$ with $|\Lambda_n|=\frac{t_n}{2},$ the sequence $(u_{\Lambda_n})_{n\in\N}$ converges (strongly) to $u$ in $MR(V,H)\cap C(0,T,V)$ as $n
\rightarrow0$ uniformly on the initial datas $(x_0,f).$ Thanks to (\ref{eq: thm Arendt-Monniaux}), such a null sequence exists .  If, in addition, we assume  that
\begin{equation*}\label{condition supp pour convergence uniforme}
\lim\limits_{t\to 0}\frac{\omega(t)}{t^{\gamma/2}}=0,
\end{equation*}
then we show that $(u_\Lambda)_{\Lambda}$ converges to $u$ uniformly on the initial datas $(u_0,f)$ in $MR(V,H)\cap C(0,T,V)$  as $|\Lambda|\to 0$ for arbitrary uniform subdivision $\Lambda$ of $[0,T].$ More precisely, we obtain that
\begin{equation*}
\|u-u_\Lambda\|_{MR}\leq \textbf{c}\Big[\omega(2|\Lambda|)+\frac{\omega(2|\Lambda|)}{|\Lambda|^{\gamma/2}}+\int_0^{2|\Lambda|}\frac{\omega(s)}{s^{1+\gamma/2}}{\rm d}s\Big]\Big[\|u_0\|_V+\|f\|_{L^2(0,T;H)}\Big],
\end{equation*}
and
\begin{equation*}
\|u-u_\Lambda\|_{C(0,T,V)}\leq\textbf{c}\Big[\omega(2|\Lambda|)+\frac{\omega(2|\Lambda|)}{|\Lambda|^{\gamma/2}}+\int_0^{2|\Lambda|}\frac{\omega(s)}{s^{1+\gamma/2}}{\rm d}s\Big]\Big[\|u_0\|_V+\|f\|_{L^2(0,T;H)}\Big]
\end{equation*}
for some positive constant $\textbf{c}>0$ depending only on $M,\alpha,\gamma$ and $c_H,$ where $c_H$ is the continuous embedding constant of $V$ into $H.$ For this we first prove that $(u_\Lambda)_\Lambda$ converges in $MR(V,V')$ uniformly on the initial datas $(u_0,f)$ as $|\Lambda|\to 0.$ This will be proved in Section \ref{Banach space case} in a more general situation.
\medskip 
\par\noindent It is well known that the solution of a non-autonomous linear evolution equation can be given by a strongly continuous evolution family $\{U(t,s): 0\leq s\leq t\leq T\}\in \L(H).$ Our approximation approach will allows us to study  whether or  not  this evolution family is eventually norm continuous. This will be the subject of a future work.
%

\section{Uniform convergence on Banach spaces \label{Banach space case}}
In this section we consider a more general setting. Namely, let $(D,\Vert .\Vert_D)$ and $(X,\Vert .\Vert)$ be
 two Banach spaces such that $D$ is continuously and densely embedded into $X$ (we write $D \underset d \hookrightarrow X$)  and let $A:[0,T]\to \mathcal{L}(D,X)$ be a strongly measurable and bounded function. Let $p\in(1,\infty)$ be fixed.
 \begin{definition}\label{def maximal regularity}
 We say that $A$ has {\rm
$L^p$-maximal regularity} on the bounded interval $[0,T],$ and
we write $A\in \mathcal{MR}_p(0,T),$ if for each interval
$[a,b]\subset[0,T]$ and  every $f\in L^p(a,b ; X)$ there exists
a unique function $u$ belonging to $L^p(a,b; D)\cap W^{1,p}(a,b; X)$ such that 
\begin{equation}\label{CPNA-x}
   \dot{u}(t)+A(t)u(t)=f(t)\ \  \hbox{a.e. on}\ [a,b],\  u(a)=0.
\end{equation}
\end{definition}

Note  that $W^{1,p}(a,b;X)\subset C([a,b];X),$ so that $u(a)=0$ in (\ref{CPNA-x}) is well defined. The maximal regularity space \[MR_p(D,X):=MR_p(a,b, D,X):=L^p(a,b; D)\cap W^{1,p}(a,b; X)\] is a Banach space for the norm
\[\|u\|_{MR}: =\|u\|_{L^p(a,b ;D)}+\|u\|_{ W^{1,p}(a,b ; X)}.\]
Definition \ref{def maximal regularity} can be reformulate in terms of sum methods. For this, we denote by $MR_{0}(a,b,D,X)$ the closed subspace of $MR_p(a,b,D,X)$ consisting
of all functions $u$ that satisfies $u(a)=0.$ For each $[a,b]\subset [0,\tau]$ consider the two unbounded linear operators $\frA=\frA_{a,b}$ and
$\frB=\frB_{a,b}$ with domains $D(\frA)=L^p(a,b;D)$ and $D(\frB)=\{u\in
W^{1,p}(a,b; X), u(a)=0\}$ defined by
\[(\frA f)(t)=A(t)f(t)\ \ \hbox{   and  } \ \ (\frB u)(t)=\dot{u}(t) \ \hbox{ for almost every } t\in[a,b].\]
\par\noindent Thus $A: [0,T]\to \mathcal{L}(D,X)$ has $L^p-$maximal regularity if and only if
the unbounded operator $\frA+\frB$ with domain $D(\frA+\frB)=MR_0(D,X)$ is invertible.

\begin{remark}\label{remark:restriction+translation} (i)\ Assume that $A\in \mathcal{MR}_p(0,T).$ Then the uniqueness of solvability in each subinterval
$[a,b]$ implies that $(\mathcal{A}_{a,b}+\mathcal{B}_{a,b})^{-1}$
is the restriction to $L^p(a,b;X)$ of $(\mathcal{A}_{0,T}+\mathcal{B}_{0,T})^{-1}.$
\par (ii)\ Remark that  $A\in\mathcal{MR}_p(D,X)$ if and only if
 $\rho +A\in \mathcal{MR}_p(D,X)$  for some (or all) $\rho\in \mathbb{C}.$ In fact, if $f\in L^p(a,b; X),$ $\rho\in \mathbb{C}$ and
 $g(t):= e^{\rho t }f(t).$ Then a function $u\in MR_p(D,X)$ satisfies
\begin{equation*}\label{eq9}\dot{u}(t)+A(t)u(t)+\rho u(t)=f(t),\ \hbox{a.e. on}\ [a,b], \ \
u(a)=0\end{equation*} if and only if $v(\cdot):=e^{\rho \cdot }u(\cdot)\in MR_p(D,X)$
satisfies \begin{equation*}\label{eq10}\dot{v}(t)+A(t)v(t)=g(t),\ \hbox{a.e.
on } [a,b], \ \ v(a)=0. \end{equation*}
\end{remark}
\bigskip 

\par If $A\in \mathcal{MR}_p(0,T),$ then for all
$0\leq a\leq b\leq T$ the homogeneous problem
\begin{equation}\label{homog}
   \dot{u}(t)+A(t)u(t)=f(t)\ \  \hbox{a.e. on}\ [a,b],\  u(a)=x
\end{equation}
has a unique solution
 $u\in MR_p(D,X)$ for all $f\in L^p(a,b; X)$ and for all $x$ in
 the \textit{ trace space }
 $$ Tr=Tr_p(a,b,D,X):=\{u(a),\ u\in MR_p(a,b,D,X)\}.$$
  The trace space is a Banach space with the norm
\[\|x\|_{Tr}:=\inf \left\lbrace \|u\|_{MR}: u(a)=x\right\rbrace.\]
Note that the trace space does not depend on the interval $[a,b].$
 It is isomorphic to the real interpolation space
 $(X,D)_{\frac{1}{p*},p},$ where $\frac{1}{p*}+\frac{1}{p}=1$. Moreover, $$MR_p(D,X)\underset{d}{\hookrightarrow}C([a,b];Tr).$$ The reader may consults e.g., \cite{ACFP}, \cite{Pr-Sc} \cite{Po-St} and the references therein for further references.

\medskip
For autonomous Cauchy problems, that is if $A(.)= A$ is constant, $L^p$-maximal regularity is independent of the bounded interval $[0,T]$ and of $p\in(1,\infty)$ \cite{Ku-We, Ca-Ve, Sob}. Further, if $A$ has $L^p$-maximal regularity then $A$ is closed as unbounded operator on $X$ and
$-A$  generates a holomorphic $C_0$-semigroup $(T(t))_{t \geq 0}$
 on $X$ \cite{A-B3, Do, Ku-We}. In Hilbert spaces an operator $A$ has $L^p$-maximal
 regularity if and only if $-A$ generates a holomorphic $C_0$-semigroup \cite{DSi}.
 This equivalence is restricted to Hilbert spaces \cite{Ka-La}, see also \cite{Fa14}. In this section we will denote by $\mathcal{MR}$ the set of all  operators
  $A\in \mathcal{L}(D,X)$ having $L^p$-maximal regularity.
\\
\par\noindent Now we recall that a strongly measurable function $A:[0,T]\longrightarrow \L(D,X)$ is (uniformly) relatively continuous (in the
  sense of \cite[Definition 2.5]{ACFP}) if for every $\varepsilon>0$  there exist $\delta> 0$ and $\eta\geq
  0$ such that for all $x\in D$ and for all $t,s\in[0,T]$ one has
  \begin{equation}\label{Defn:relative continuity}
  \Vert A(t)x-A(s)x\Vert\leq \varepsilon \Vert x \Vert_D+
  \eta\Vert x\Vert\
	\end{equation}  whenever $\vert t-s\vert \leq \delta.$ Note that if $A$ is
relatively continuous then $A$ is bounded. The notion of relative continuity was introduced
in  by  Arendt, Chill, Fornaro and Poupaud,
to establish $L^p$-maximal regularity \cite[Theorem 2.7]{ACFP}. 

%

Next we assume that there exists an
approximation $A_n: [0,\tau]\longmapsto \mathcal{L}(D,X)$ (strongly
measurable)  of $A$ with  the following
properties:\vskip2mm
\begin{itemize}
 \item[$(H_1)$] there exists $C>0$ such that
    $\|A_n(t)\|_{\mathcal{L}(D,X)}\leq C$\ \  for all $t\in[0,\tau]$
    and $n\in \mathbb{N},$
\vskip2mm

 \item[$(H_2)$] for each $x\in D$
    one has $ A_n(t)x \to A(t)x$ as $n\to \infty$ in $X$ $t-$a.e.
on $[0,\tau],$
 \vskip2mm \item[$(H_3)$] for every $\varepsilon>0$
there exist
     $\eta\geq 0,\  n_0\in\mathbb{N}$ such that for all $x\in
     D,
     n\geq n_0, t\in[0,\tau]$ one has
    \begin{equation}\label{Defn:relative continuity approx}
		\Vert A_n(t)x-A(t)x\Vert\leq \varepsilon
     \Vert x\Vert_D+\eta\Vert x\Vert.
		\end{equation}
\vskip2mm
\item[$(H_4)$]  $A_n\in
    \mathcal{MR}_p(0,T)$ for all $n\in \mathbb{N}.$\end{itemize}
Then the following stability result was proved by EL-Mennaoui and Laasri \cite[Theorem 4.5]{ELLA13}.
\begin{theorem}\label{staglob} Let $A: [0,T]\longrightarrow
\mathcal{L}(D,X)$ be strongly measurable and relatively
continuous. Assume that $A(t)\in \mathcal{MR}$ for all
$t\in[0,T]$ and $A_n$ satisfy the hypothesis $(H_1)-(H_4)$. Let
$u_n\in Tr$ and $f_n\in L^p(0,T;X)$ such that
$x_n\longrightarrow x$ in $Tr$ and
  $f_n\longrightarrow f$ in $L^p(0,T;X).$ Then the sequence $(u_n)_{n\in\N}$ of solutions of 
	\begin{equation}\label{eqq1}
  \dot{u}_n(t)+A_n(t)u_n(t)=f_n(t) \quad{a.e} \quad\text
   on \quad [0,T],\
 \ \ \ \ u_n(0)=x_n \
\quad \end{equation} converges in $MR_p(D,X)$ and
$u:=\lim\limits_{n\to \infty }u_n$ is the unique solution of
\begin{equation}\label{eqq2}
  \dot{u}(t)+A(t)u(t)=f(t) \quad{a.e} \quad\text on \quad [0,T],\
 \ \ \ \ u(0)=x.\
\end{equation}
\end{theorem}
The aim of this section is to show that for $x=x_n=0$ the convergence established   in  Theorem \ref{staglob} is actually uniform with respect the the inhomogeneity $f.$ If $\eta=0$ in (\ref{Defn:relative continuity approx}), then we obtain that such a convergence is uniform with respect to  both initial datas $f$ and $x$.
\begin{theorem}\label{Thm: uniform convergence for inhomogenuous case in Banach space} Let $A: [0,T]\longrightarrow
\mathcal{L}(D,X)$ be strongly measurable and relatively
continuous. Assume that $A(t)\in \mathcal{MR}$ for all
$t\in[0,T]$ and $A_n$ satisfy the hypothesis $(H_1)-(H_4)$. Then for every $\varepsilon>0$ there exists $n_0\in\N$ such that for all $n\geq n_0$ one has
\begin{equation}\label{uniform estimate for inhomogenuous case}\|(\frA+\frB)^{-1}-(\frA_n+\frB)^{-1}\|_{\L(L^p(0,T,X), MR_p(D,X))}\leq \varepsilon. \end{equation}
\end{theorem}
\begin{proof} We proceed in three steps and follow the same idea as in the proof of \cite[Theorem 4.5]{ELLA13}.
	
\par \textit{Step 1.}
By \cite[Lemma 4.1]{ELLA13} there exists a constant $M(A)>0$ and $\rho_1\geq 0$ independent on $t\in[0,T]$  such that

\begin{equation}\label{estimation RM1}\|(\rho+\mathfrak{A}(t)+\frB)^{-1}\|_{\mathcal{L}(L^p(a,b;X),
MR_p(a,b,D,X))}\leq M(A)
\end{equation}
and
\begin{equation}\label{estimation RM2}
\|(\rho+\frA(t)+\frB)^{-1}\|_{\mathcal{L}(L^p(a,b,X))}\leq
\frac{M(A)}{1+\rho},
\end{equation}
 for all $\rho\geq\rho_1$ and all $[a,b]\subset [0,T].$ On the other hand, we have from \cite[Lemma 4.2]{ELLA13} that there exists $\rho_2\geq 0, \delta>0$ and $n_0\in \N$ such that for each $[a,b]\subset [0,T],|b-a|\leq\delta$ implies that
 \begin{equation}\label{3/4estimation} \|(\frA_n-\frA(t))(\rho+\frA(t)+
\frB)^{-1}\|_{\mathcal{L} (L^p(a,b;X))}\leq 3/4,\end{equation}
for all
$t\in[0,T], n\geq n_0$ and all $\rho\geq\rho_2.$ Since $A$ satisfies the assumptions $(H1)-(H4),$  we also have that
\begin{equation}\label{3/4estimation2} \|(\frA-\frA(t))(\rho+\frA(t)+
\frB)^{-1}\|_{\mathcal{L} (L^p(a,b;X))}\leq 3/4,\end{equation}
 for all
$t\in[0,T]$ and all $\rho\geq\rho_2$ provided that $|b-a|\leq\delta.$
\medskip
\par\noindent $\textit{Step 2.}$ Let $\delta>0, \rho_0:=\max\{\rho_1,\rho_2\}\geq 0$ and $n_0\in\N$ be as in the first step and assume that $T\leq\delta.$ Let  $t_0\in [0,T]$ and $\rho>\rho_0$ be fixed.
Let $\varepsilon>0$ and let $k_0\in \mathbb{N}$ be such that
\begin{equation}\label{eq2}
\sum_{k=k_0+1}^{\infty}\|\Big((\frA_n-\frA(t_0))
(\rho+\frA(t_0)+\frB)^{-1}
\Big)^k\|_{\mathcal{L}(L^p(0,T;X)}\leq \frac{\varepsilon}{3M(A)}
\end{equation}
and \begin{equation}\label{eq3}
\sum_{k=k_0+1}^{\infty}\|\Big((\frA-\frA(t_0))
(\rho+\frA(t_0)+\frB)^{-1}
\Big)^k\|_{\mathcal{L}(L^p(0,T;X)}\leq \frac{\varepsilon}{3M(A).}
\end{equation}
For each $k\in \{1,...,k_0\}$ and $n\in \mathbb{N}$ with $n\geq n_0$ we set
\[ I_{k,n}:=\Big((\frA_n-
\frA(t_0))(\rho+\frA(t_0)+\frB)^{-1} \Big)^k
\text{ and  }\ I_{k}:=\Big((\frA-
\frA(t_0))(\rho+\frA(t_0)+\frB)^{-1} \Big)^k.\]
 According to  $(H3), (\ref{estimation RM1})$ and $(\ref{estimation RM2}),$  there exists $n_1\in\N$ and $\eta\geq 0$ such that for each $n\geq N_0:=\max\{n_0,n_1\}$
\begin{align*}
\|I_{1,n}&f-I_1f\|_{L^p(0,T,X)}=\|(\frA_n-\frA)(\rho+\frA(t_0)+\frB)^{-1}f\|_{L^p(0,T,X)}
\\&\leq \frac{\varepsilon^{\prime}}{2 M(A)}\|(\rho+\frA(t_0)+\frB)^{-1}f\|_{MR_p(D,X)}
+\eta\|(\rho+\frA(t_0)+\frB)^{-1}f\|_{L^p(0,T,X)}
\\&\leq \frac{\varepsilon^{\prime}}{2}\|f\|_{L^p(0,T,X)}+\frac{\eta M(A)}{1+\rho}\|f\|_{L^p(0,T,X)}
\end{align*}
 where $\varepsilon':=\frac{4\varepsilon}{9M(A)k_0^2}$. Thus choosing  $\rho\geq \rho_0$ large enough we obtain
\[\|I_{1,n}f-I_1f\|_{L^p(0,T,X)}\leq  \varepsilon^{\prime}\|f\|_{L^p(0,T,X)}\]
for all $n\geq N_0.$ This estimate together with $(\ref{3/4estimation})$ and $(\ref{3/4estimation2}),$ yield
\begin{align*}\|I_{2,n}f-I_2f\|_{L^p(0,T,X)}&=\|I_{1,n}I_{1,n}f-I_1I_1f\|_{L^p(0,T,X)}
\\&\leq \|I_{1,n}(I_{1,n}-I_1)f\|_{L^p(0,T,X)}+\|(I_{1,n}-I_{1})I_1f\|_{L^p(0,T,X)}
\\&\leq \frac{3}{4}\varepsilon^{\prime}\|f\|_{L^p(0,T,X)}+\frac{3}{4}\varepsilon^{\prime}\|f\|_{L^p(0,T,X)}=\frac{3}{2}\varepsilon^{\prime}\|f\|_{L^p(0,T,X)},
\end{align*}
and thus
\begin{equation}\label{eq I_k}\|I_{k,n}f-I_kf\|_{L^p(0,T,X)}\leq \frac{3}{4}k_0\varepsilon^{\prime}\|f\|_{L^p(0,T,X)}=\frac{\varepsilon}{3M(A)}\|f\|_{L^p(0,T,X)}\end{equation}
 holds for all $n\geq N_0$ and every $k=1,2,..,k_0.$ Combining (\ref{eq2}), (\ref{eq3})  and (\ref{eq I_k}) we deduce
\begin{align*}
\|(\frA_n+\frB)^{-1}f-&(\frA+\frB)^{-1}f\|_{MR}
\\&=\|(\rho+\frA(t_0)+\frB)^{-1}\Big(I+(\frA_n-\frA(t_0))(\rho+\frA(t_0)+\frB)^{-1}\Big)^{-1}f
\\&\quad\ \ \ \ -(\rho+\frA(t_0)+\frB)^{-1}\Big(I+(\frA-\frA(t_0))
(\rho+\frA(t_0)+\frB)^{-1}\Big)^{-1}f\|_{MR_p(D,X)}
\\&\leq \|(\rho+\frA(t_0)+\frB)^{-1}\|_{\L(L^p(0,T,X),MR_p(D,X)}\|\sum_{k=1}^\infty(I_{k,n}-I_k)f\|_{L^p(0,T,X)}
\\&\leq M(A)\Big(\frac{\varepsilon}{3M(A)}+\frac{\varepsilon}{3M(A)}+\frac{\varepsilon}{3M(A)}\Big)\|f\|_{L^p(0,T,X)}=\varepsilon\|f\|_{L^p(0,T,X)}
\end{align*}
for all $n\geq N_0.$
\medskip
\par\noindent \textit{Step 2.} Let now $[0,T]$ be an arbitrary closed and bounded interval and set
\[\tau:=\max\{0\leq \tau'\leq T \ \text{ such that } (\ref{uniform estimate for inhomogenuous case}) \text{ holds on } [0,\tau']\}.\]
Then $\tau\geq \delta.$ We show that $\tau=T.$ Assume by contradiction that $\tau<T$ and let $\tau'_0<\tau$ such that $\tau-\tau_0'\leq \delta/2.$ Then (\ref{uniform estimate for inhomogenuous case}) holds if we consider the Cauchy problems (\ref{eqq1}) and (\ref{eqq2})  on $[0,\tau'_0]$ or on $[\tau'_0,(\tau'_0+\delta)\wedge T],$ and thus on $[0,(\tau'_0+\delta)\wedge T]$ by taking into account Remark \ref{remark:restriction+translation} $(i)$. Thus $(\tau'_0+\delta)\wedge T\leq\tau,$ which is a contradiction. This completes the proof.
\end{proof}

The main result of this section is the follwing.
\begin{theorem}\label{Thm: uniform convergence for homogenuous case in Banach space} Let $A: [0,T]\longrightarrow
\mathcal{L}(D,X)$ be strongly measurable and relatively
continuous. Assume that $A(t)\in \mathcal{MR}$ for all
$t\in[0,T]$ and $A_n$ satisfy the hypothesis $(H_1)-(H_4)$. Let
$x\in Tr$ and $f\in L^p(0,T;X).$ Let $u_n,u\in MR_p(D,X)$ be, respectively,  the solution of
\begin{equation}\label{eqq11}
  \dot{u}_n(t)+A_n(t)u_n(t)=f_n(t) \quad{a.e} \quad\text
   on \quad [0,T],\
 \ \ \ \ u_n(0)=x \
\quad \end{equation}
and
\begin{equation}\label{eqq21}
  \dot{u}(t)+A(t)u(t)=f(t) \quad{a.e} \quad\text on \quad [0,T],\
 \ \ \ \ u(0)=x.\
\end{equation}
Then for each $\varepsilon>0$ there exists $n_0\in\N$ and $\tilde\eta>0$ such that for all $n\geq n_0$
\[\|u_n-u\|_{MR}\leq \varepsilon\Big[ \|x\|_{Tr}+ \|f\|_{L^p(0,T;X)}\Big]+\tilde\eta\eta  \|x\|_{Tr}.\]
\end{theorem}
\begin{proof}
Let $\varepsilon>0.$ Choose $\vartheta\in MR_p(D,X)$ such that $\vartheta(0)=x$ and $\|\vartheta\|_{MR}\leq 2\|x\|_{Tr}.$ Set  $g_n:=-\dot{\vartheta}(\cdot)-A_n(\cdot)\vartheta(\cdot)+f(\cdot)$ and
$g:=-\dot{\vartheta}(\cdot)-A(\cdot)\vartheta(\cdot)+f(\cdot)\in L^p(0,T;X).$ Then there exist $v_n,u\in MR_p(D,X)$ such that
 \begin{equation*}
  \dot{v}_n(t)+A_n(t)v_n(t)=g_n(t)\quad{a.e} \quad\text on \quad [0,T],\
 \ \ \ \ v_n(0)=0 \
\end{equation*}
and 
 \begin{equation*}
  \dot{v}(t)+A(t)v(t)=g(t) \quad{a.e} \quad\text on \quad [0,T],\
 \ \ \ \ v(0)=0. \
\end{equation*}
By the uniqueness of solvability, $u_n=v_n+\vartheta$ and $u=v+\vartheta.$ It follows,
\begin{align*}
\|u_n-u\|_{MR}=&\|v_n-v\|_{MR}=\|(\frA_n+\frB)^{-1}g_n-(\frA+\frB)^{-1}g\|_{MR}
\\&\leq \|(\frA_n+\frB)^{-1}(g_n-g)\|_{MR}+\|(\frA_n+\frB)^{-1}g-(\frA+\frB)^{-1}g\|_{MR}
\\&\leq M\|A_n\vartheta-A\vartheta\|_{L^p(0,T;X)}+\|(\frA_n+\frB)^{-1}g-(\frA+\frB)^{-1}g\|_{MR}
\end{align*}
where \[M:=\sup_{n\in \N}\|(\frA_n+\frB)^{-1}\|_{\L(L^p(0,T,X),MR_p(D,X))}\]
which is finite by the uniform boundedness principal.
Next, Theorem \ref{Thm: uniform convergence for inhomogenuous case in Banach space} and condition $(H3)$ imply that there exists $n_0\in\N$ and $\eta\geq 0$ such that
\begin{align*}\|A_n\vartheta-A\vartheta\|_{L^p(0,T;X)}&\leq \frac{\varepsilon}{4M}\|\vartheta\|_{MR}+\eta \|\vartheta\|_{L^p(0,T;X)}
\leq \frac{\varepsilon}{2M}\|x\|_{Tr}+2\eta \|x\|_{Tr}
\end{align*}
and
\begin{align*}\|(\frA_n+\frB)^{-1}g-(\frA+\frB)^{-1}g\|_{MR}\leq & \frac{\varepsilon}{4c+2}\|g\|_{L^p(0,T;X)}
\\&\leq\frac{\varepsilon}{4c+2}(\|\dot \vartheta+A \vartheta\|_{L^p(0,T;X)}+\|f\|_{L^p(0,T;X)})
\\&\leq \frac{\varepsilon}{2}\|x\|_{Tr}+\varepsilon \|f\|_{L^p(0,T;X)}
\end{align*}
for all $n\geq n_0$, where $c=\max\{1,\sup_{t\in[0,T]}\|A(t)\|_{L(D,X)}\}.$ This shows the claims.
\end{proof}
\section{Non-autonomous forms: assumptions and preliminary results \label{Approximation}}

Throughout the following sections $H,V$ are two separable  Hilbert spaces over $\mathbb C$ such that  $V \underset d \hookrightarrow H;$ i.e., $V$ is  densely embedded into $H$ and
\begin{equation*}\label{eq:V_dense_in_H}
    \|u\| \le c_H \|u\| _V \quad (u \in V)
\end{equation*}
for some constant $c_H>0.$ Let $V'$ denote  the antidual of $V.$  The duality
between $V'$ and $V$ is denoted by $\langle ., . \rangle$. As usual, by identifying $H$ with  $H',$ we have $V\hookrightarrow H\cong H'\hookrightarrow V'.$  These embeddings are continuous and
\begin{equation*}\label{eq:H_dense_in_V'}
    \|f\|_{V'} \le c_H \|f\|  \quad (f \in V')
\end{equation*}
see e.g., \cite{Bre11}. We denote by $(\cdot \mid \cdot)_V$ the scalar product and $\|\cdot\|_V$ the norm
on $V$ and by $(\cdot \mid \cdot), \|\cdot\|$ the corresponding quantities in $H.$  Let $
    \fra: [0,T]\times V\times V \to \C$ be a  non-autonomous sesquilinear form satisfying
\begin{equation}\label{eq:continuity-nonaut}
    |\fra(t,u,v)| \le M \Vert u\Vert_V \Vert v\Vert_V \quad (t\in[0,T],u,v\in V)\qquad
\end{equation}
and
\begin{equation}\label{eq:Ellipticity-nonaut}
    \Re ~\fra (t,u,u)\ge \alpha \|u\|^2_V \quad ( t\in [0,T], u\in V)
\end{equation}
for some constants  $\alpha, M> 0$ and $\fra(.,u,v)$ is measurable for all $u,v\in V.$
We assume in addition, that there exists $0\leq \gamma< 1$ and a non-decreasing continuous  function   $\omega:[0,T]\longrightarrow [0,+\infty)$ with
\begin{equation}\label{eq 1:Dini-condition}  \sup_{t\in[0,T]} \frac{\omega(t)}{t^{\gamma/2}}<\infty, \end{equation}
\begin{equation}\label{eq 2:Dini-condition}
\int_0^T\frac{\omega(t)}{t^{1+\gamma/2}} {\rm d}t<\infty
\end{equation}
and
\begin{equation}\label{eq 3:Dini-condition}
    |\fra(t,u,v)-\fra(s,u,v)| \le\omega(|t-s|) \Vert u\Vert_{V} \Vert v\Vert_{V_\gamma}
\end{equation}
for all $t,s\in[0,T]$ and for all $u,v\in V$ where $V_\gamma:=[H,V]_\gamma$ is the complex interpolation space. Note that

\[V\hookrightarrow V_\gamma \hookrightarrow H\hookrightarrow V_\gamma'\hookrightarrow V'\] with continuous  embeddings. By the Lax-Milgram theorem, for each $t\in[0,T]$ there exists an isomorphism $\A(t):V\to V^\prime$ such that
$\langle \A(t) u, v \rangle = \fra(t,u,v)$ for all $u,v \in V.$ It is well known that $-\A(t),$
regarding as unbounded operator with domain $V,$  generates a bounded holomorphic semigroup $e^{-\cdot\A(t)}$  of angle $\theta:=\frac{\pi}{2}-\arctan(\frac{M}{\alpha})$ on $V'$. We call $\A(t)$ the operator associated with $\fra(t,\cdot,\cdot)$ on $V^\prime.$ We have also to consider the operator $A(t)$ associated with $\fra(t,\cdot,\cdot)$ on  $H:$
\begin{align*}
    D(A(t)) := {}& \{ u\in V : \A(t) u \in H \}\\
    A(t) u = {}& \A(t) u.
\end{align*}
Then  $-A(t)$ generates a holomorphic $C_0$-semigroup (of angle $\theta$) $e^{-s A(t)}$ on $H$ which is the restriction to $H$ of $e^{-\cdot A(t)},$ and we have
\begin{equation}\label{analytic representation}
e^{-\cdot A(t)}=\frac{1}{2i\pi}\int_\Gamma e^{\cdot\mu} (\mu+A(t))^{-1}{\rm d}\mu
\end{equation}
where $\Gamma:=\{re^{\pm \varphi}:\ r>0\}$ for some fixed $\varphi\in (\theta,\frac{\pi}{2})$
(see e.g. \cite[Lecture 7]{Ar06},\cite{Ka},\cite[Chapter 1]{Ou05} or \cite[Chap.\ 2]{Tan79}).
\\\par The following proposition is of great interest for this paper.

\begin{proposition}\label{lemma: estimations for general from}\cite[Section 2]{Ar-Mo15} Let $b$ be any sesquilinear form that satisfies assumptions (\ref{eq:continuity-nonaut})-(\ref{eq:Ellipticity-nonaut}) with the same constants $M$ and $\alpha$ and let $\gamma\in [0,1[.$ Let $\B$ and $B$ be the associated operators  on $V'$ and $H,$ respectively. Then there exists a constant $c>0$ which depends only on $M,\alpha, \gamma$ and $c_H$ such that

\begin{enumerate}
\item \label{Eq1: estimation resolvent}$\displaystyle \|(\lambda-\B)^{-1}\|_{\L(V_{\gamma}',H)}\leq \frac{c}{(1+\mid\lambda\mid)^{1-\frac{\gamma}{2}}},$
\item \label{Eq2: estimation resolvent}$\displaystyle \|(\lambda-\B)^{-1}\|_{\L(V)}\leq \frac{c}{1+\mid\lambda\mid},$
\item \label{Eq3: estimation resolvent}$\displaystyle \|(\lambda-\B)^{-1}\|_{\L(H,V)}\leq \frac{c}{(1+\mid\lambda\mid)^{\frac{1}{2}}},$
\item \label{Eq4: estimation resolvent} $\displaystyle \|(\lambda-\B)^{-1}\|_{\L(V',H)}\leq \frac{c}{(1+\mid\lambda\mid)^{\frac{1}{2}}},$
\item\label{Eq6: estimation resolvent} $\displaystyle \|(\lambda-\B)^{-1}\|_{\L(V'_\gamma,V)}\leq \frac{c}{(1+\mid\lambda\mid)^{\frac{1-\gamma}{2}}},$
\item \label{Eq2: estimation  semigroup}
$\displaystyle\|e^{-s\B}\|_{\L(V_\gamma',H)}\leq\frac{c}{s^{\gamma/2}},$
\item \label{Eq3: estimation  semigroup}
$\displaystyle\|e^{-sB}\|_{\L(V_\gamma',V)}\leq\frac{c}{s^{\frac{1+\gamma}{2}}},$
\item \label{Eq4: estimation  semigroup}
$\displaystyle\|e^{-s\B}\|_{\L(V',V)}\leq\frac{c}{s^{\frac{1}{2}}},$
\item \label{analytic estimation}
$\displaystyle\|Be^{-sB}\|_{\L(H)}\leq \frac{c }{s},$
\item \label{analytic estimation in V}
$\displaystyle\|e^{-sB}\|_{\L(V)}\leq c$
\end{enumerate}
for each $t\in[0,T], s\geq 0$ and $\lambda\notin \Sigma_\theta:=\{re^{i\varphi}: r>0, |\varphi|<\theta\}.$
\end{proposition}

\par Let $\Lambda=(0=\lambda_0<\lambda_1<...<\lambda_{n+1}=T)$ be
a uniform subdivision of $[0,T],$ i.e., \[|\Lambda|:=\sup_l|\lambda_{l+1}-\lambda_l|=|\lambda_{k+1}-\lambda_k| \  \text{ for each } \ k=0,1,...,n.\] Consider a family of sesquilinear  forms $\fra_k:V \times V \to \C$  given by  
\begin{equation}\label{eq:form-moyen integrale}
 \fra_k(u,v):=\frac{1}{\lambda_{k+1}-\lambda_k}
\int_{\lambda_k}^{\lambda_{k+1}}\fra(r;u,v){\rm  d}r,\quad  u,v\in V \ \
\end{equation}
for each $k=0,1,...,n.$ Remark that $\fra_k$ satisfies (\ref{eq:continuity-nonaut}) and (\ref{eq:Ellipticity-nonaut}) for all $k=0,1,...n.$ The associated operators are denoted by $\A_k\in \L(V,V')$ and are given by
\begin{equation}\label{eq:op-moyen integrale}
 \A_ku :=\frac{1}{\lambda_{k+1}-\lambda_k}
\int_{\lambda_k}^{\lambda_{k+1}}\A(r)u{\rm  d}r,\quad u\in V, \ k=0,1,...,n.\ \  \end{equation}
The function $\fra_\Lambda:[0,T]\times V \times V \to \C$ defined for $t\in [\lambda_k,\lambda_{k+1}]$ by
\begin{equation}\label{form: approximation formula2}
 \fra_{\Lambda}(t;u,v):=\frac{\lambda_{k+1}-t}{\lambda_{k+1}-\lambda_k}\fra_{k}(u,v)+\frac{t-\lambda_k}{\lambda_{k+1}-\lambda_k}\fra_{k+1}(u,v), \ \ u,v\in V,
\end{equation}
 is a non-autonomous sesquilinear form which  satisfies   (\ref{eq:continuity-nonaut})-(\ref{eq:Ellipticity-nonaut}) with the same constants $\alpha$ and $M.$  The associated time dependent operator is denoted by
  \begin{equation}\label{Operator: approximation formula1}\A_\Lambda(.):[0,T]\to \L(V,V')
\end{equation}
 and is given for $t\in [\lambda_k,\lambda_{k+1}]$ by
\begin{equation}\label{Operator: approximation formula2}\
 \A_{\Lambda}(t):=\frac{\lambda_{k+1}-t}{\lambda_{k+1}-\lambda_k}\A_{k}+\frac{t-\lambda_k}{\lambda_{k+1}-\lambda_k}\A_{k+1}.
\end{equation}
Then  $\A_\Lambda$ converges
strongly and almost everywhere to $\A$ and also on $\L(L^2(0,T,V),L^2(0,T,V'))$  as $\vert \Lambda\vert
\rightarrow 0$ \cite[Lemma 2.1]{ELLA15}.
\begin{remark}\label{Remark: estimations for general from} All estimates in Proposition \ref{lemma: estimations for general from} holds for $\A_\Lambda(t)$ with constant independent of $\Lambda$ and $t\in[0,T],$ since $\fra_\Lambda$ satisfies (\ref{eq:continuity-nonaut})-(\ref{eq:Ellipticity-nonaut}) with the same constants $M$ and $\alpha,$ also $\gamma$ and $c_H$  does not depend on $\Lambda$ and $t\in[0,T].$
\end{remark}
\noindent Recall that a coercive and bounded form $b:V\times V\to \C$ associated with the operator $B$ on $H$ has the Kato square root property if
\begin{equation}\label{squar root properties} D(B^{1/2})=V.\end{equation}
We prove below that $\fra_\Lambda(t,\cdot,\cdot)$ has the square root property for all $t\in[0,T]$ if $\fra_\Lambda(0;\cdot,\cdot)$ has it.  This is essentially based on the abstract result  due to Arendt and Monniaux \cite[Proposition 2.5]{Ar-Mo15}. They proved that for two sesquilinear forms $\fra_1,\fra_2:V\times V\to\C$ which satisfies (\ref{eq:continuity-nonaut})-(\ref{eq:Ellipticity-nonaut}), the form $\fra_1$ has the square root property if and only if $\fra_2$ has it provided that
\[|\fra_1(u,v)-\fra_2(u,v)|\leq c \|u\|_V\|v\|_{V_\gamma} \ u,v\in V\]
for some constant $c>0.$

\begin{proposition}\label{square root property for the Linear-approximation} Assume $\fra(0,.,.)$ has the square root property. Then $\fra_\Lambda(t,.,.)$ has the square root properties for all $t\in [0,T],$ too.
\end{proposition}
\begin{proof} Let $t\in [0,T]$ and let $k\in\{0,1,\cdots,n\}$ be such that $t\in [\lambda_k,\lambda_{k+1}].$ Then Then assumption (\ref{eq 3:Dini-condition}) implies that
\begin{align*}\mid \fra_\Lambda(t,u,v)-&\fra(0,u,v)\mid\leq\frac{1}{\lambda_{k+1}-\lambda_k}
\int_{\lambda_k}^{\lambda_{k+1}}\mid\fra(r;u,v)-\fra(0,u,v)\mid{\rm  d}r
\\&\qquad \qquad \ \ \ +\frac{1}{\lambda_{k+2}-\lambda_{k+1}}
\int_{\lambda_{k+1}}^{\lambda_{k+2}}\mid\fra(r;u,v)-\fra(0,u,v)\mid{\rm  d}r
\\&\leq \frac{1}{\lambda_{k+1}-\lambda_k}
\int_{\lambda_k}^{\lambda_{k+1}}\omega(r)\|u\|_V\|v\|_{V_\gamma}{\rm  d}r
+\frac{1}{\lambda_{k+2}-\lambda_{k+1}}
\int_{\lambda_{k+1}}^{\lambda_{k+2}}\omega(r)\|u\|_V\|v\|_{V_\gamma}{\rm  d}r
\\&\leq 2\sup_{t\in[0,T]}\omega(t)\|u\|_V\|v\|_{V_\gamma}.
\end{align*}
Now the claim follows from \cite[Proposition 2.5]{Ar-Mo15}.\end{proof}
\noindent
The following results will play an important role latter in the study of the convergence. We first prove that $\fra_\Lambda$ has also a modulus of continuity of the same art as for $\fra.$ In what follows we extend $\omega$ to $[0,2T]$ by setting $\omega(t)=\omega(T)$ for $T\leq t\leq 2T.$
\begin{proposition}\label{Prop: Dini condition for Linear-approximation}
For all $u,v\in V,\ t,s\in [0,T]$
\begin{equation}\label{eq:Dini-condition for Linear-approximation}
    |\fra_\Lambda(t,u,v)-\fra_\Lambda(s,u,v)| \le\omega_\Lambda(|t-s|) \Vert u\Vert_{V} \Vert v\Vert_{V_\gamma}
\end{equation}
where $\omega_\Lambda:[0,T]\longrightarrow [0,+\infty[$ is defined by
\[ \omega_\Lambda
(t):=\left\{%
\begin{array}{ll}
    \frac{t}{|\Lambda|}\omega(4|\Lambda|)& \hbox{for } 0\leq t\leq 2|\Lambda|,\\
    2\omega(2t) & \hbox{ for } 2|\Lambda|<t\leq T. \\
\end{array}%
\right. \]
Moreover, \begin{equation}\label{eq 2:Dini-condition for Linear-approximation}
\int_0^T\frac{\omega_\Lambda(s)}{s^{1+\gamma/2}}{\rm  d}s\leq \frac{4}{1-\frac{\gamma}{2}}\sup_{t\in[0,T]} \frac{\omega(t)}{t^{\gamma/2}}+2^{\gamma/2}\int_{0}^{2T}\frac{\omega(s)}{s^{1+\gamma/2}}{\rm  d}s<\infty,
\end{equation}
and
 \begin{equation}\label{eq 3:Dini-condition for Linear-approximation}
\sup_{t\in[0,T]}\frac{\omega_\Lambda(t)}{t^{\gamma/2}}\leq 2^{1+\gamma/2}\sup_{t\in[0,T]}\frac{\omega(t)}{t^{\gamma/2}}<\infty.
\end{equation}
\end{proposition}
\begin{proof} Let $u,v\in V$ and $t,s\in [0,T].$  For the proof of (\ref{eq:Dini-condition for Linear-approximation}) we distinguish three cases
\par\noindent\textit{Case 1:} If $\lambda_k\leq s<t \leq\lambda_{k+1}$ for some fixed $k\in\{0,1,\cdots,n\}.$ Then we obtain, using (\ref{eq 3:Dini-condition}) and the fact that $\omega$ is non-decreasing, that
\begin{align*}
|\fra_\Lambda(t,u,v)-\fra_\Lambda(s,u,v)|
&=\Big\vert\frac{\lambda_{k+1}-t}{\lambda_{k+1}-\lambda_k}\fra_{k}(u,v)+\frac{t-\lambda_k}{\lambda_{k+1}-\lambda_k}\fra_{k+1}(u,v)
\\&\qquad\qquad\qquad\ \ \ \ -\frac{\lambda_{k+1}-s}{\lambda_{k+1}-\lambda_k}\fra_{k}(u,v)-\frac{s-\lambda_k}{\lambda_{k+1}-\lambda_k}\fra_{k+1}(u,v)\Big\vert
\\&=\frac{(t-s)}{|\Lambda|}\Big\vert\fra_{k}(u,v)-\fra_{k+1}(u,v)\Big\vert
\\&\leq \frac{(t-s)}{|\Lambda|}\frac{1}{|\Lambda|}\int_{0}^{|\Lambda|}\mid a(r+\lambda_{k},u,v)-a(r+\lambda_{k+1},u,v)\mid{\rm  d}r
\\&\leq \frac{(t-s)}{|\Lambda|}\frac{1}{|\Lambda|}\int_{0}^{|\Lambda|}\omega(\lambda_{k+1}-\lambda_{k})\Vert u\Vert_{V} \Vert v\Vert_{V_\gamma}{\rm  d}r
=\frac{(t-s)}{|\Lambda|}\omega(|\Lambda|)\Vert u\Vert_{V} \Vert v\Vert_{V_\gamma}
\end{align*}
\par\noindent \textit{Case 2:} If $\lambda_k\leq s\leq \lambda_{k+1} \leq t\leq\lambda_{k+2},$  then we deduce from \textit{Step 1} that
\begin{align*}
|\fra_\Lambda(t,u,v)-\fra_\Lambda(s,u,v)|&\leq
|\fra_\Lambda(t,u,v)-\fra_\Lambda(\lambda_{k+1},u,v)|+|\fra_\Lambda(\lambda_{k+1},u,v)-\fra_\Lambda(s,u,v)|
\\&\leq \frac{t-\lambda_{k+1}}{|\Lambda|}\omega(|\Lambda|)\Vert u\Vert_{V} \Vert v\Vert_{V_\gamma}+
\frac{\lambda_{k+1}-s}{|\Lambda|}\omega(|\Lambda|)\Vert u\Vert_{V} \Vert v\Vert_{V_\gamma}
\\&=\frac{t-s}{|\Lambda|}\omega(|\Lambda|)\Vert u\Vert_{V} \Vert v\Vert_{V_\gamma}.
\end{align*}
\textit{Case 3:}
 If now $\lambda_k\leq s\leq \lambda_{k+1}<\cdots<\lambda_{l}\leq t\leq\lambda_{l+1}.$ Then $\lambda_l-\lambda_{k+1}\leq t-s\leq \lambda_{l+1}-\lambda_{k}$ and thus
\begin{equation}\label{Eq1 Proof Proposition: Dini-condition for Linear-approximation}|t-s +\lambda_{k+1}-\lambda_{l+1}|\leq |\Lambda|.
\end{equation} It follows that
\begin{align*}
\fra_\Lambda&(t,u,v)-\fra_\Lambda(s,u,v)\\&=
\frac{\lambda_{l+1}-t}{\lambda_{l+1}-\lambda_l}\fra_{l}(u,v)+\frac{t-\lambda_l}{\lambda_{l+1}-\lambda_l}\fra_{l+1}(u,v)
-\frac{\lambda_{k+1}-s}{\lambda_{k+1}-\lambda_l}\fra_{k}(u,v)-\frac{s-\lambda_k}{\lambda_{k+1}-\lambda_k}\fra_{k+1}(u,v)
\\&=\frac{\lambda_{l+1}-t}{|\Lambda|}[\fra_l(u,v)-\fra_k(u.v)]+\frac{t-\lambda_{l}}{|\Lambda|}[\fra_{l+1}(u,v)-\fra_{k+1}(u.v)]
\\&\qquad+\frac{\lambda_{l+1}-\lambda_{k+1}+s-t}{|\Lambda|}\fra_k(u,v)+\frac{\lambda_{k}-\lambda_{l}+t-s}{|\Lambda|}\fra_{k+1}(u,v)
\end{align*}
Because of (\ref{Eq1 Proof Proposition: Dini-condition for Linear-approximation}) and since $\lambda_k-\lambda_l=\lambda_{k+1}-\lambda_{l+1},$ we deduce that
\begin{align*}
\mid\fra_\Lambda(t,u,v)-\fra_\Lambda(s,u,v)\mid&\leq \frac{\lambda_{k+1}-t}{|\Lambda|}\omega(\lambda_l-\lambda_k)+\frac{t-\lambda_{k}}{|\Lambda|}\omega(\lambda_{l+1}-\lambda_{k+1})
\\&\qquad\qquad \ \ \ \ \ \ \ \ +\frac{\mid t-s+\lambda_{l+1}-\lambda_{k+1}\mid}{|\Lambda|}\omega(\lambda_{l+1}-\lambda_{l})
\\&\leq\omega(\lambda_l-\lambda_k)+\omega(\lambda_{l+1}-\lambda_l)
\\&\leq 2\omega(2(t-s)).
\end{align*}
This completes the proof of (\ref{eq:Dini-condition for Linear-approximation}). 
\newline Let now  prove (\ref{eq 2:Dini-condition for Linear-approximation}). By construction we have
\begin{align*}\int_0^T\frac{\omega_\Lambda(t)}{t^{1+\gamma/2}}{\rm  d}t&
=\int_0^{2|\Lambda|}\frac{\omega(4|\Lambda|)}{|\Lambda|}t^{-\gamma/2}{\rm  d}t+\int_{2|\Lambda|}^T\frac{\omega(2t)}{t^{1+\gamma/2}}{\rm  d}t
\\&\leq\frac{2^{\frac{\gamma}{2}+1}}{1-\frac{\gamma}{2}}\frac{\omega(4|\Lambda|)}{(4|\Lambda|)^{\gamma/2}}+2^{\gamma/2}\int_{0}^{2T}\frac{\omega(t)}{t^{1+\gamma/2}}{\rm  d}t
\\&\leq \frac{4}{1-\frac{\gamma}{2}}\sup_{t\in[0,T]} \frac{\omega(t)}{t^{\gamma/2}}+2^{\gamma/2}\int_{0}^{2T}\frac{\omega(t)}{t^{1+\gamma/2}}{\rm  d}t
\end{align*}
which is finite by (\ref{eq 1:Dini-condition}). The inequality (\ref{eq 3:Dini-condition for Linear-approximation}) is easy to prove.
\end{proof}
Note that condition (\ref{eq 3:Dini-condition}) implies that $\A(t)-\A(s)\in \L(V,V_\gamma')$ for each $t,s\in[0,T]$  and
\begin{equation}\label{Eq0: Dini condition operators}
\|\A(t)-\A(s)\|_{\L(V,V_\gamma')}\leq \omega(|t-s|).
\end{equation}
According to Proposition \ref{Prop: Dini condition for Linear-approximation}, similar estimates hold for $\A_\Lambda(\cdot):$
\begin{lemma}\label{Lemma:Dini condition approximation operators} For each  $t,s\in [0,T]$ we have  $\A_\Lambda(t)-\A_\Lambda(s)\in \L(V,V_\gamma'),$
\begin{equation}\label{Eq: Dini condition operators}
\|\A_\Lambda(t)-\A_\Lambda(s)\|_{\L(V,V_\gamma')}\leq \omega_\Lambda(|t-s|)
\end{equation}
and \begin{equation}\label{Eq2:Dini condition operators}
\|\A_\Lambda(t)-\A(t)\|_{\L(V,V_\gamma')}\leq 2\omega(2|\Lambda|).
\end{equation}

\end{lemma}
\begin{proof}
The estimate (\ref{Eq: Dini condition operators}) follows from (\ref{eq:Dini-condition for Linear-approximation}).
For the second statement, let $t\in[0,T]$ and let $k\in\{0,1,\cdots,n\}$ be such that $t\in[\lambda_k,\lambda_{k+1}].$ Then
\begin{align*}
\A_\Lambda(t)-\A(t)&=\frac{\lambda_{k+1}-t}{\lambda_{k+1}-\lambda_k}[\A_k-\A(t)]+
\frac{t-\lambda_{k}}{\lambda_{k+1}-\lambda_k}[\A_{k+1}-\A(t)]
\\&=\frac{\lambda_{k+1}-t}{(\lambda_{k+1}-\lambda_k)^2}\int_{\lambda_k}^{\lambda_{k+1}}[\A(r)-\A(t)]{\rm d}r+
\frac{t-\lambda_{k}}{(\lambda_{k+1}-\lambda_k)^2}\int_{\lambda_{k+1}}^{\lambda_{k+2}}[\A(r)-\A(t)]{\rm d}r.
\end{align*}
Then using (\ref{Eq0: Dini condition operators}) and the fact that $\omega$ is non-decreasing we obtain
\begin{align*}
\|\A_\Lambda(t)-\A(t)\|_{\L(V,V_\gamma')}
&\leq \frac{\lambda_{k+1}-t}{(\lambda_{k+1}-\lambda_k)^2}\int_{\lambda_k}^{\lambda_{k+1}}\omega(t-r){\rm d}r+
\frac{t-\lambda_{k}}{(\lambda_{k+1}-\lambda_k)^2}\int_{\lambda_{k+1}}^{\lambda_{k+2}}\omega(t-r){\rm d}r
\\&\leq \omega(|\Lambda|)+\omega(2|\Lambda|)\leq 2\omega(2|\Lambda|),
\end{align*}
which proves the claim.\end{proof}
\section{$L^2$-maximal regularity in $H:$ a weak approximation}\label{section 3}
Recall that  $V,H$ denote two separable Hilbert spaces and $\fra:[0,T]\times V\times V\to \C$ is a non-autonomous  form satisfying  (\ref{eq:continuity-nonaut})-(\ref{eq 3:Dini-condition}) such that $D(A(0)^{1/2})=V.$ Let $\A(t)$ the operator associated with  $\mathfrak a(t,.,.)$  on $V'$ for each $t\in [0,T]$ and consider the  non-autonomous Cauchy problem
\begin{equation}\label{Abstract Cauchy problem}
\dot{u}(t)+\A(t)u(t)=f(t) \quad  {a.e.} \ \ \text on \quad [0,T], \ \ u(0)=u_0
\end{equation}
Let $\Lambda$ be an uniform subdivision of $[0,T],$ \[\A_\Lambda: [0,T]\to \L(V,V') \quad \text{and}\quad \fra_\Lambda :[0,T]\times V\times V\to \C\] be given by (\ref{Operator: approximation formula1})-(\ref{Operator: approximation formula2}) and  (\ref{form: approximation formula2}), respectively, and consider the Cauchy problem
 \begin{equation}\label{nCP in V'}
\dot{u}_\Lambda (t)+\A_\Lambda(t)u_\Lambda(t)=f(t) \quad  {a.e.} \ \ \text on \quad [0,T], \ \ u_\Lambda(0)=u_0.
\end{equation}
 Clearly, $t\mapsto\fra_\Lambda(\cdot,u,v)$ is piecewise $C^1$ for all $u,v\in V.$ Moreover, $\fra_\Lambda$ has the Kato square property by  Lemma \ref{square root property for the Linear-approximation}. Then the Cauchy problem (\ref{nCP in V'}) has  $L^2-$maximal regularity in $H$ by \cite[Theorem 1.1]{Bar71} or \cite[Theorem 4.2]{ADLO14}. On the other hand, we known by Lions' theorem that  for a given $f\in L^2(0,T,H)$ and $u_0\in V$ the Cauchy problem (\ref{Abstract Cauchy problem}) has a unique solution $u\in MR(V,V^\prime)$. Furthermore, it is known that the seqeunce $(u_\Lambda)_{\Lambda}$ of solutions of (\ref{nCP in V'}) converges (strongly) in $MR(V,V')$ to $u$ as $|\Lambda|$ goes to $0$ \cite[Proposition 3.2]{ELLA15}. The  main result of this section show that  $(u_\Lambda)$ converges weakly in $MR(V,H)$  to $u$  as $|\Lambda|$ goes to $0.$ This in particular gives an alternative proof of Theorem \ref{thm: Arendt-Monniaux}.
\begin{theorem}\label{main theorem:l2 max reg in H} Let $f\in L^2(0,T;H)$ and $u_0\in V$ and let $u_\Lambda\in MR(V,H)$ be the
    solution of (\ref{nCP in V'}). Then $u_\Lambda$ converges weakly in $MR(V,H)$ as $|\Lambda|
    \longrightarrow 0$ and $u:={\rm w}-\lim\limits_{|\Lambda|\to 0}u_\Lambda$ satisfies
		(\ref{Abstract Cauchy problem}).
\end{theorem}
\noindent For the proof we need first some preliminary lemmas. Let $f\in L^2(0,T;H)$ and $u_0\in V,$ then  the solution $u_\Lambda$ of (\ref{nCP in V'}) satisfies the following key formula
\begin{equation}\label{Formula for the solution}
u_\Lambda(t)=e^{-tA_\Lambda(t)}u_0+\int_{0}^te^{-(t-s)\A_\Lambda(t)}f(s){\rm  d}s+\int_{0}^te^{-(t-s)A_\Lambda(t)}(\A_\Lambda(t)-\A_\Lambda(s))u_\Lambda(s){\rm  d}s
\end{equation}
for all $t\in[0,T].$ This formula is due to Acquistapace and Terreni \cite{Ac-Ter87} and was proved in a more general setting in \cite[Proposition 3.5]{Ar-Mo15}. For the operator valued function $\A_\Lambda,$ this formula can be derived in a
 more classical way. In the sequel we will use the following notations:

\begin{equation}\label{Formula for the solution 2}
u_{\Lambda,1}(t):=e^{-tA_\Lambda(t)}u_0, \qquad  u_{\Lambda,2}(t):=\int_{0}^t e^{-(t-s)A_\Lambda(t)}f(s){\rm  d}s
\end{equation}
and

\begin{equation}\label{Formula for the solution 3}
u_{\Lambda,3}(t):=\int_{0}^te^{-(t-s)\A_\Lambda(t)}(\A_\Lambda(t)-\A_\Lambda(s))u_\Lambda(s){\rm  d}s.
\end{equation}
 The next two lemmas follow, thanks to Proposition \ref{Prop: Dini condition for Linear-approximation}, Lemma \ref{Lemma:Dini condition approximation operators} and  Remark \ref{Remark: estimations for general from}, by using the same argument as in the proof of Arendt and Monniaux \cite[Theorem 4.1]{Ar-Mo15}.

\begin{lemma}\label{Lemma: Invertibility of I-QLambda}
Let $Q_\Lambda^\mu:L^2(0,T,H)\to L^2(0,T,H)$ denotes the linear operator  defined for all $g\in L^2(0,T,H)$ and $\mu\geq 0$ by
\begin{equation}\label{def l operator Q}
(Q_\Lambda^\mu g)(t):=\int_{0}^t(\A_\Lambda(t)+\mu)e^{-(t-s)(\A_\Lambda(t)+\mu)}(\A_\Lambda(t)-\A_\Lambda(s))(\A_\Lambda(s)+\mu)^{-1}g(s){\rm  d}s\ \quad t\textrm{-a.e}.
\end{equation}
Then $\lim\limits_{\mu \to \infty}\|Q_\Lambda^\mu\|_{L^2(0,T,H)}=0$ uniformly on $\Lambda$ and thus $I-Q_\Lambda^\mu$ is invertible on $L^2(0,T,H)$ for $\mu$ large enough and for all $\Lambda.$
\end{lemma}
\begin{lemma} \label{Lemma: The uniform  Boundedness in $L^2(0,T,H)$} There exists constant $c>0$ depending only on $\alpha,M,\gamma$ and $c_H$  such that
\begin{align}\label{eq1: uniform estimate of u1}
\|\A_\Lambda u_{\Lambda,1}\|_{L^2(0,T,H)}&\leq  c \|u_0\|_{V}^2,
\\\label{eq2: uniform estimate of u1}
     \|\A_\Lambda u_{\Lambda,2}\|_{L^2(0,T,H)}&\leq c\|f\|_{L^2(0,T,H)}.
\end{align}
\end{lemma}

\par According to Lemma \ref{Lemma: Invertibility of I-QLambda} and replacing $\A_\Lambda(t)$ with $\A(t)_\Lambda+\mu,$ we may assume  without loss of generality that $Q_\Lambda=Q^\mu_\Lambda$ satisfies $\|Q_\Lambda\|_{\L(L^2(0,T,H))}< 1,$ and then $I-Q_\Lambda$ is invertible by the Neumann series. Now we can give the proof of Theorem \ref{main theorem:l2 max reg in H}.

%
\begin{proof}(\textit{of Theorem} \ref{main theorem:l2 max reg in H}) Since $(I-Q_\Lambda)$ is invertible in  $L^2(0,T,H),$ we deduce from (\ref{Formula for the solution}) that
\[\dot u_\Lambda=\A_\Lambda u_\Lambda=(I-Q_\Lambda)^{-1}(\A_\Lambda u_\Lambda^1+\A_\Lambda u_\Lambda^2).\]
This equality and Lemma \ref{Lemma: The uniform  Boundedness in $L^2(0,T,H)$},  yield the estimate
\begin{equation}
\Vert \dot u_\Lambda\Vert_{L^2(0,T;H)}\leq {\bf c}\big[\Vert u_0\Vert_V+\Vert f\Vert_{L^2(0,T;H)}\big]
\end{equation}
for a constant ${\bf c}>0$ independent of the subdivision $\Lambda.$ Since for all $t\in [0,T]$ one has $u_\Lambda(t)=u_\Lambda(0)+\int_0^t \dot u_\Lambda(s){\rm d}s,$ we conclude that
\begin{equation}
\Vert u_\Lambda\Vert_{H^1(0,T;H)}\leq {\bf c}\big[\Vert u_0\Vert_V+\Vert f\Vert_{L^2(0,T;H)}\big]
\end{equation}
for some constant ${\bf c}>0$ independent of the subdivision $\Lambda.$
Then there exists a subsequence of $(u_\Lambda),$ still denoted by $(u_\Lambda)$ that converges weakly to some $v\in H^1(0,T,H)$ as $|\Lambda|\longrightarrow 0.$
\par\noindent We known that the Cauchy problem (\ref{Abstract Cauchy problem}) has a unique solution $u\in MR(V,V')$ by Lions' theorem. On the other hand, $(u_\Lambda)$ converges strongly to $u$ on $MR(V,V')$  by \cite[Proposition 3.2]{ELLA15}. In particular, $u_\Lambda\longrightarrow u$ in $L^2(0,T,V).$ Thus $\A_\Lambda u_\Lambda\longrightarrow Au$ in $L^2(0,T,V')$ as $|\Lambda|\longrightarrow 0$ by \cite[Lemma 2.1]{ELLA15}. It follows, by the uniqueness of the limits , that $u= v\in H^1(0,T,H)$ since  \[\dot u_\Lambda=f-\A_\Lambda u_\Lambda\underset{|\Lambda|\rightarrow 0}{\longrightarrow} f-\A u=\dot u \ \text{ in } L^2(0,T,V').\]
 This completes the proof.
\end{proof}

\section{$L^2$-maximal regularity in $H:$ uniform approximation \label{Uniform Convergence}}

Assume that  $H,V$  and $\fra:[0,T]\times V\times V\longrightarrow \C$ are as in Section \ref{Approximation}. Let $(f,u_0)\in L^2(0,T,H)\times V,$  $\Lambda$ be a uniform subdivision of $[0,T]$ and $u, u_\Lambda\in MR(V,H)$ be the solutions of (\ref{Abstract Cauchy problem}) and (\ref{nCP in V'}) respectively. In the previous section we have seen  that $(u_\Lambda)$ converges weakly to $u$ in $MR(V,H)$ as $|\Lambda|\longrightarrow 0.$ The aim of this
 section is to prove that this convergence holds for the strong topology of $MR(V,H)$ and uniformly on the initial data $u_0$ and $f.$ 
\\

The following result is the key idea of this  section.
\begin{theorem}\label{convergence uniform} There exists a positive constant $\textbf{c}>0$ depending only on $M,\alpha,\gamma$ and $c_H$  such that 
\begin{equation}\label{estimation of the error of the convergence2}
\|u-u_\Lambda\|_{H^1(0,T,H)}\leq \textbf{c}\Big[\omega(2|\Lambda|)+\frac{\omega(2|\Lambda|)}{|\Lambda|^{\gamma/2}}
+\int_0^{2|\Lambda|}\frac{\omega(t)}{t^{1+\gamma/2}}{\rm d}t
\Big]\Big[\|f\|_{L^2(0,T,H)}+\|u_0\|_V\Big].
\end{equation}

\end{theorem}
With this estimate theorem in hand, the study of the uniform convergence, with respect to  initial datats,
 of $u_\Lambda \longrightarrow u$ in $MR(V,H)$ becomes easy. Due to the results of Section \ref{Banach space case} and  hypothesis that $\omega$ satisfies all we need is to look when
\begin{equation}\label{Condition supplementaire 1}
\lim\limits_{|\Lambda|\to 0}\frac{\omega(2|\Lambda|)}{|\Lambda|^{\gamma/2}}=0
\end{equation}
holds. Endeed, clearly  (\ref{Eq0: Dini condition operators})  implies that $\A:[0,T]\longrightarrow\L(V,V')$ is, in particular, continuous. Moreover, $\A_\Lambda:[0,T]\longrightarrow\L(V,V')$ satisfies conditions $(H_1)$-$(H_4)$ (see  Section \ref{Banach space case}) by taking $D=V$ and  $X=V'.$ Thus one can apply Theorem \ref{Thm: uniform convergence for homogenuous case in Banach space} and conclude that $u_\Lambda \longrightarrow u$ in $L^2(0,T,V)$ uniformly on the initial data $u_0\in V\subset H=Tr(V,V')$ and the homogeneity $f\in L(0,T,H).$
\begin{corollary}\label{Corollary 1: uniform convergence} There exists a null sequence $(t_n)_{n\in\N}\subset [0,T]$ depending on $\omega$ such that  for every $\varepsilon>0$ there exists $n_0\in\N$ such that for all $n\geq n_0$ one has
\[\|u-u_{\Lambda_n}\|_{MR}\leq \varepsilon \big[\|u_0\|_V+\|f\|_{L^2(0,T,H)}\big].\]
for all subdivisions $\Lambda_n$ of $[0,T]$ with $2|\Lambda_n|=t_n.$
\end{corollary}
\begin{proof} We claim that \[\liminf_{t\to 0} \frac{\omega(t)}{t^{\gamma/2}}=0.\] Other wise the integral
\[\int_0^T \frac{\omega(s)}{s^{1+\gamma{2}}}{
\rm d}s=\infty\]
which contradict the assumption (\ref{eq 2:Dini-condition}). This and  Theorem \ref{convergence uniform} completes the proof.
\end{proof}
Finally, if we assume that $\omega$ satisfies the following addition condition
\begin{equation}\label{Condition supplementary 2}
\lim\limits_{t\to 0}\frac{\omega(t)}{t^{\gamma/2}}=0,
\end{equation}
then the statement of Corollary \ref{Corollary 1: uniform convergence} holds for all uniform subdivision $\Lambda$ of $[0,T].$ 
\begin{corollary} \label{Corollary 2: uniform convergence}For all $\varepsilon>0$ there exists $\delta>$ such that for each subdivision $\Lambda$ of $[0,T]$
\[|\Lambda|\leq \delta \Longrightarrow \|u-u_{\Lambda}\|_{MR}\leq \varepsilon \big[\|x_0\|_V+\|f\|_{L^2(0,T,H)}\big].\]
\end{corollary}
Now we give the proof of Theorem \ref{convergence uniform}:
\begin{proof}(\textit{Proof of Theorem \ref{convergence uniform}}.) We will use the representation formula (\ref{Formula for the solution}), notations (\ref{Formula for the solution 2})-(\ref{Formula for the solution 3}) and the the corresponding quantities for the solution $u$ of  (\ref{Abstract Cauchy problem}).  We proceed by several steps.
\par $a)$ First, we estimate  $\A_\Lambda u_{\Lambda,1}-\A u_{1}$  in $L^2(0,T,H).$ Let $t\ne 0.$ We obtain using the second estimate in Proposition \ref{Lemma:Dini condition approximation operators} and the estimates (\ref{Eq2: estimation  semigroup}) and  (\ref{analytic estimation in V}) in Proposition \ref{lemma: estimations for general from} that
\begin{align}
\nonumber\|\A_\Lambda(t) u_{\Lambda,1}(t)-&\A(t)u_{1}(t)\|_H=\|\A_\Lambda(t)e^{-t \A_\Lambda(t)}u_0-\A(t)e^{-t \A(t)}u_0\|_H
\\\nonumber&\leq  \|e^{-t \A_\Lambda(t)}[\A_\Lambda(t)u_0-\A(t)u_0]\|_H+
\|[e^{-t \A_\Lambda(t)}-e^{-t \A(t)}]\A_\Lambda(t)u_0\|_H
\\\nonumber&=\|e^{-t \A_\Lambda(t)}[\A_\Lambda(t)u_0-\A(t)u_0]\|_H+\int_0^t \|e^{-(t-s) \A_\Lambda(t)}(\A_\Lambda(t)-\A(t))e^{-s \A(t)}u_0\|_H
\\&\label{eq1: proof of Theorem on convergence uniform}  \leq 2c\omega(2|\Lambda|)\left(\frac{1}{t^{\gamma/2}}+c\int_0^t\frac{1}{s^{\gamma/2}}\rm{ ds}\right)\|u_0\|_{V}.
\end{align}
Similarly, combining the estimates (\ref{Eq1: estimation resolvent}) and (\ref{Eq3: estimation resolvent}) in Proposition \ref{lemma: estimations for general from} and  the estimate (\ref{Eq2:Dini condition operators}) in Proposition \ref{lemma: estimations for general from} we obtain
\begin{align}\nonumber
\|\A_\Lambda(t) &u_{\Lambda,2}(t)-\A(t)u_{2}(t)\|_H
\\\nonumber\leq\int_0^t& \|[\A_\Lambda(t)e^{-(t-s)A_\Lambda(t)}-\A(t)e^{-(t-s)A(t)}]f(s)\|_Hds
\\\nonumber&\leq\frac{1}{2\pi}\int_{0}^t\int_{\Gamma}\mid\lambda\mid e^{-(t-s)\Re\lambda}\|(\lambda-\A_\Lambda(t))^{-1}(\A_\Lambda(t)-\A(t))(\lambda-\A(t))^{-1}f(s)\|_{H}{\rm  d}\lambda{\rm  d}s
									\\\nonumber&\leq \frac{1}{\pi}\int_{0}^t\omega(2|\Lambda|)\int_{\Gamma}e^{-(t-s)\Re\lambda}
                  \frac{c^2}{\mid\lambda\mid^{\frac{1-\lambda}{2}}}\|f(s)\|_{H}{\rm  d}\lambda{\rm  d}s
									\\\nonumber&=\frac{c^2\omega(2|\Lambda|)}{\pi}\int_{0}^t \|f(s)\|_{H}\int_{0}^\infty
                  \frac{e^{-(t-s)r\cos(\nu)}}{r^{\frac{1-\gamma}{2}}}{\rm  d}r{\rm  d}s
							\\	\nonumber&=\frac{c^2\omega_\Lambda(t)}{2\pi}\int_{0}^t \|f(s)\|_{H}\int_{0}^\infty\frac{e^{-\rho\cos(\theta)}}{(\frac{\rho}{t-s})^{-\frac{1+\gamma}{2}}}{{\rm  d}\rho}{\rm  d}s
\\&=\label{eq2: proof of Theorem on convergence uniform}
\frac{c^2\omega(2|\Lambda|) }{\pi}\int_{0}^\infty
                  \frac{e^{-\rho\cos(\nu)}}{\rho^{\frac{1-\gamma}{2}}}{\rm  d\rho}\int_{0}^t\|f(s)\|_{H}(t-s)^{-\frac{1+\gamma}{2}}{\rm  d}s.
\end{align}
The last integral is well defined since the function $h:\R\to \R$ given by $h(t)=t^{-\frac{1+\gamma}{2}}$ for $t\in]0,T]$ and $h(t)=0$ for $t\in ]-\infty,0]\cap]T,+\infty[$  belongs to $L^1(\R)$ because $\frac{1+\gamma}{2}<1.$ The estimates (\ref{eq1: proof of Theorem on convergence uniform}) and (\ref{eq2: proof of Theorem on convergence uniform}) yield, respectively,
\begin{equation}
\|\A_\Lambda u_{\Lambda,1}-\A u_1\|_{L^2(0,T,H)}\leq \textbf{c} \omega(2|\Lambda|)\|u_0\|_{V}
\end{equation}
and
\begin{equation}\label{estimation error Au2-Au2lambda}
\|\A_\Lambda u_{\Lambda,2}-\A u_2\|_{L^2(0,T,H)}\leq \textbf{c} \omega(2|\Lambda|)\|f\|_{L(0,T,H)}
  \end{equation}
for a positive constant $\textbf{c}>0$  that depends only on $M,\alpha,\gamma$ and $c_H.$
\par\noindent $b)$ Next, we  prove  the following estimate
\begin{equation}\label{estimation of the error of the convergence for Q}
\|Q_\Lambda -Q\|_{\L(L^2(0,T,H))}\leq c\Big[\omega(2|\Lambda|)+\frac{\omega(2|\Lambda|)}{|\Lambda|^{\gamma/2}}+\int_0^{2|\Lambda|}\frac{\omega(s)}{s^{1+\gamma/2}}\Big]
\end{equation}
 where $Q:L^2(0,T,H)\longrightarrow L^2(0,T,H)$ is defined via formula which is analogous to (\ref{def l operator Q}). To this end,  for  $g\in L^2(0,T,H)$ and $t\in [0,T]$ we write 
	
	\begin{align*}\label{eq1: proof strong conv Qlambda}
	\|(Q_\Lambda g)(t)&-(Q g)(t)\|_H\\&\leq \int_0^t\|\A_\Lambda(t)e^{-(t-s)\A_\Lambda(t)}(\A_\Lambda(t)-\A_\Lambda(s))(A^{-1}_\Lambda(s)-A^{-1}(s))g(s)\|_H ds
	\\&\quad + \int_0^t\|\A_\Lambda(t)e^{-(t-s)\A_\Lambda(t)}(\A_\Lambda(t)-\A(t)-\A_\Lambda(s)+\A(s))A^{-1}(s)g(s)\|_H ds
	\\&\quad + \int_0^t\|(\A_\Lambda(t)e^{-(t-s)\A_\Lambda(t)}-\A(t)e^{-(t-s)A(t)})(\A(t)-\A(s))A^{-1}(s)g(s)\|_H ds
\\&=I_{\Lambda,1}(t)+I_{\Lambda,2}(t)+I_{\Lambda,3}(t)
	\end{align*}
Replacing $\A(s)$ by $\A(s)+\mu$ and according to Proposition \ref{lemma: estimations for general from} we may assume $\|\A^{-1}_\Lambda(s)\|_{\L(V_\gamma',V)}\leq c$ and $\|\A^{-1}_\Lambda(s)\|_{\L(H,V)}\leq c.$ Next, by the estimates (\ref{Eq2: estimation  semigroup}) and (\ref{analytic estimation}) in Proposition \ref{lemma: estimations for general from}  together with   (\ref{Eq: Dini condition operators}) and (\ref{Eq2:Dini condition operators}), we have
\begin{align*}
I_{\Lambda,1}(t)=&\int_0^t\|\A_\Lambda(t)e^{-\frac{t-s}{2}\A_\Lambda(t)}e^{-\frac{t-s}{2}\A_\Lambda(t)}(\A_\Lambda(t)-\A_\Lambda(s))
(\A^{-1}_\Lambda(s)-\A^{-1}(s))g(s)\|_H ds
\\&\leq 2^{1+\gamma/2}c^2\int_{0}^t \frac{\omega_\Lambda(t-s)}{(t-s)^{1+\gamma/2}}\|(\A^{-1}_\Lambda(s)-\A^{-1}(s))g(s)\|_Vds
\\&=2^{1+\gamma/2}c^2\int_0^t\frac{\omega_\Lambda(t-s)}{(t-s)^{1+\gamma/2}}\|(\A^{-1}_\Lambda(s)(\A_\Lambda(s)-\A(s))\A^{-1}(s))g(s)\|_Vds
\\&\leq2^{3+\gamma/2}c^2 \omega(2|\Lambda|)\int_0^t \frac{\omega_\Lambda(t-s)}{(t-s)^{1+\gamma/2}}\|\A^{-1}_\Lambda(s)\|_{\L(V_\gamma',V)}\|\A^{-1}(s)\|_{\L(H,V)}\|g(s)\|_{H}ds
\\&\leq 2^{3+\gamma/2}c^4\omega(2|\Lambda|)\int_0^t \frac{\omega_\Lambda(t-s)}{(t-s)^{1+\gamma/2}}\|g(s)\|_{H}ds
\\&=2^{3+\gamma/2}c^4\omega(2|\Lambda|) (h_\Lambda\ast\|g(\cdot)\|_{H})(t),
\end{align*}
where $h_{\Lambda}(t):=\omega_\Lambda(t)t^{-1-\gamma/2}$ for $t\in[0,T]$ and $h_{\Lambda}(t):=0$ for $t\in (-\infty,0[\cap]T,+\infty).$  Proposition \ref{Prop: Dini condition for Linear-approximation} implies  that  $h_\Lambda\in L^1(\R)$ and that $\| h_\Lambda\|_{L^1(\R)}$ is bounded uniformly with respect to the subdivision $\Lambda.$ Therefore we obtain
\begin{equation}
\int_0^T I_{\Lambda,1}^2(s){\rm d}s\leq \textbf{c} \omega(2|\Lambda|)^2\int_0^T\|g(s)\|_{H}^2ds
\end{equation}
where the  positive constant $\textbf{c}>0$ is independent of $\Lambda.$
\par\noindent Again using as above the estimates (\ref{Eq2: estimation  semigroup}) and (\ref{analytic estimation}) in Proposition \ref{lemma: estimations for general from}, we obtain for the second term $I_{\Lambda,2}$
\begin{align*}
&I_{\Lambda,2}(t):=\int_0^t\|\A_\Lambda(t)e^{-(t-s)\A_\Lambda(t)}(\A_\Lambda(t)-\A(t)-\A_\Lambda(s)+\A(s))A^{-1}(s)g(s)\|_H ds
\\&\leq 2^{1+\gamma/2}c^3 \int_{0}^t \|(\A_\Lambda(t)-\A(t)-\A_\Lambda(s)+\A(s))\|_{\L(V_\gamma',H)}
\frac{\|g(s)\|_{H}}{(t-s)^{1+\gamma/2}}{\rm  d}s
\\&\leq 2^{1+\gamma/2}c^3\int_{0}^t \kappa_\Lambda(t-s)\|g(s)\|_{H}{\rm  d}s
\end{align*}
where
\begin{equation*}\label{Kappa}\kappa_{\Lambda}
(t):=\left\{%
\begin{array}{ll}
     [\omega(t)+\omega_\Lambda(t)]t^{-(1+\frac{\gamma}{2})} & \hbox{ if } \  0\leq  t<2|\Lambda|,\\
    4\omega(2|\Lambda|)t^{-(1+\frac{\gamma}{2})}& \hbox{ if } \  2|\Lambda|< t\leq 2T, \\
   0& \hbox{ if } \   t\in ]-\infty,0]\cap]2T,+\infty[. \\
\end{array}%
\right. \end{equation*}
Here we have used simultaneously both estimates (\ref{Eq: Dini condition operators}) and (\ref{Eq2:Dini condition operators}) from Lemma \ref{Lemma:Dini condition approximation operators}. Because of  (\ref{eq 2:Dini-condition}) and (\ref{eq 2:Dini-condition for Linear-approximation}), the function
$t\mapsto k_{\Lambda}(t)$ belongs to $L^1(\R)$, and by a simple calculation we obtain

\[\|\kappa_\Lambda\|_{L^1(\R)}\leq {\bf c}\Big(\frac{\omega(2|\Lambda|)}{|\Lambda|^{\gamma/2}}+\int_0^{2|\Lambda|}\frac{\omega(s)}{s^{1+\gamma/2}}{\rm d}s\Big)\]
and therefore,
\begin{equation}
\int_0^T I_{\Lambda,2}^2(s){\rm d}s\leq  {\bf c}\Big(\frac{\omega(2|\Lambda|)}{|\Lambda|^{\gamma/2}}+\int_0^{2|\Lambda|}\frac{\omega(s)}{s^{1+\gamma/2}}{\rm d}s\Big)^2\int_0^T\|g(s)\|_{H}^2ds
\end{equation}

\noindent for a constant ${\bf c}={\bf c}(M,\alpha,c_H,\gamma)>0$ independent of $\Lambda.$
\par\noindent $b)$ For the last term $I_{\Lambda,3}(t),$ we set $\tilde{g}(t,\cdot):=(\A(t)-\A(\cdot))A^{-1} (\cdot)g(\cdot).$ Again by  Lemma \ref{Lemma:Dini condition approximation operators} and  (\ref{Eq4: estimation resolvent}) and (\ref{Eq6: estimation resolvent}) from Proposition \ref{lemma: estimations for general from} and we obtain
\begin{align*}
&I_{\Lambda,3}(t):=\int_0^t\|(\A_\Lambda(t)e^{-(t-s)\A_\Lambda(t)}-\A(t)e^{-(t-s)A_\Gamma(t)})\tilde{g}(t,s)\|_H ds
\\&\leq\frac{1}{2\pi}\int_{0}^t\int_{\Gamma}\mid\lambda\mid e^{-(t-s)\Re\lambda}
                  \|(\lambda-\A_\Lambda(t))^{-1}(\A_\Lambda(t)-\A(t))(\lambda-\A(t))^{-1}\tilde{g}(t,s)\|_{H}{\rm  d}\lambda{\rm  d}s
\\&\leq \frac{1}{2\pi}\int_{0}^t\int_{\Gamma}\mid\lambda\mid e^{-(t-s)\Re\lambda}
                  \|(\lambda-\A_\Lambda(t))^{-1}\|_{\L(V',H)}\|(\A_\Lambda(t)-\A(t))(\lambda-\A(t))^{-1}\tilde{g}(t,s)\|_{V'}{\rm  d}\lambda{\rm  d}s
\\&\leq\frac{ C_{V_\gamma '}}{2\pi}\int_{0}^t\int_{\Gamma}\mid\lambda\mid e^{-(t-s)\Re\lambda}
                  \frac{c2\omega(2|\Lambda|)}{(1+|\lambda|)^{1/2}}\|(\lambda-\A(t))^{-1}\|_{\L(V_\gamma',V)}\|\tilde{g}(t,s)\|_{V_\gamma'}{\rm  d}\lambda{\rm  d}s
	\\&\leq \omega(2|\Lambda|)\frac{ C_{V_\gamma '}c^2}{\pi}\int_{0}^t\int_{\Gamma}
                  \frac{\mid\lambda\mid e^{-(t-s)\Re\lambda}}{(1+|\lambda|)^{1-\frac{\gamma}{2}}}\|\tilde{g}(t,s)\|_{V_\gamma'}{\rm  d}\lambda{\rm  d}s
	\\&\leq\omega(2|\Lambda|)\frac{ C_{V_\gamma '}c^2}{\pi}\int_{0}^t\int_{0}^\infty
                  r^{\frac{\gamma}{2}} e^{-(t-s)r\cos(\nu)}\|\tilde{g}(t,s)\|_{V_\gamma'}{\rm  d}r{\rm  d}s
\end{align*}
where $C_{V_\gamma '}$ is the injection constant of $V_\gamma '$ into $V'.$ Next, since
\[\|\tilde{g}(t,s)\|_{V_\gamma'}\leq \omega(t-s)\|A^{-1}(t)\|_{\L(H,V)}\|g(s)\|_H,\]we conclude that

\begin{align*}
I_{\Lambda,3}(t)&\leq \omega(2|\Lambda|)\frac{ C_{V_\gamma '}c^2}{\pi}\int_{0}^\infty\frac{e^{-\rho\cos(\nu)}}{\rho^{-\gamma/2}}{\rm  d}\rho \int_{0}^t\frac{\omega(t-s)}{(t-s)^{1+\frac{\gamma}{2}}}\|g(s)\|_{H}{\rm  d}s
\\&=\omega(2|\Lambda|)\frac{ C_{V_\gamma '}c^2}{\pi}\int_{0}^\infty\frac{e^{-\rho\cos(\nu)}}{\rho^{-\gamma/2}}{\rm  d}\rho(h\ast\|g(\cdot)\|_{H})(t)
\end{align*}
where $h:\R\longrightarrow \R$ is defined analogously as $h_\Lambda$ above. Taking into account (\ref{eq 2:Dini-condition}), it follows
\begin{equation}
\int_0^T I_{\Lambda,3}^2(s){\rm d}s\leq {\bf c} \omega(2|\Lambda|)^2\int_0^T\|g(s)\|_{H}^2ds
\end{equation}
for a constant $\bf c>0$ independent of $\Lambda,$ and  thus the desired estimate (\ref{estimation of the error of the convergence for Q}) is proved.
%
\par\noindent $c)$ Finally, by using  Lemma \ref{Lemma: The uniform  Boundedness in $L^2(0,T,H)$} we conclude from $a)-b)$ that
\begin{align*}
 \nonumber\|\A_\Lambda u_\Lambda-\A u\|_{L^2(0,T,H)}
&\leq \|(I-Q_\Lambda)^{-1}(\A_\Lambda u_{\Lambda,1}-\A u_{1})\|_{L^2(0,T,H)}
 \\\nonumber&\quad+\|(I-Q_\Lambda)^{-1}[\A_\Lambda u_{\Lambda,2}-\A u_{2}]\|_{L^2(0,T,H)}
 \\\nonumber&\qquad+\|(I-Q_\Lambda)^{-1}(Q-Q_\Lambda)(I-Q)^{-1}(\A u_{1}+\A u_{2})\|_{L^2(0,T,H)}
 \\&\label{estimation final 1}\leq {\bf c} \Big[\omega(2|\Lambda|)+\frac{\omega(2|\Lambda|)}{|\Lambda|^{\gamma/2}}+\int_0^{2|\Lambda|}\frac{\omega(s)}{s^{1+\gamma/2}}{\rm d}s\Big]\Big[\|u_0\|_{V}+\|f\|_{L^2(0,T,H)}\Big]
\end{align*}
where ${\bf c}>0$ is independent of $\Lambda.$ Further, since $u$ and $u_\Lambda$ satisfy  (\ref{nCP in V'}) and (\ref{Abstract Cauchy problem}), respectively,  we have
  \begin{equation}\label{estimation final 2}
   \|\dot u_\Lambda-\dot u\|_{L^2(0,T,H)} \leq c \Big[\omega(2|\Lambda|)+\frac{\omega(2|\Lambda|)}{|\Lambda|^{\gamma/2}}
	+\int_0^{2|\Lambda|}\frac{\omega(s)}{s^{1+\gamma/2}}{\rm d}s\Big]\Big[\|u_0\|_{V}+\|f\|_{L^2(0,T,H)}\Big].
	\end{equation}
	Now since $u(t)$ and $u_\Lambda(t)$ belong to $V$ for almost every $t\in [0,T],$ we have
	\[u_\Lambda(t)=u_\Lambda(0)+\int_0^t \dot u_\Lambda(s){\rm d}s \quad \text{ and  }
u(t)=u(0)+\int_0^t \dot u(s){\rm d}s\] almost everywhere. This completes the proof.
 \end{proof}

\section{Continuity of solutions}
 
%
Assume that  $H,V$ and $\fra:[0,T]\times V\times V\longrightarrow \C$ are as in the previous section. 
The aim of the this section is the prove that $(u_\Lambda)_{\Lambda}$ converges to $u$ in the space $C([0,T],V)$ uniformly on $(f,u_0)$ provided that (\ref{Condition supplementary 2}) holds. Note that $u_\Lambda\longrightarrow u$  in $C([0,T],H)$ since $MR(V,V')$ is continuously embedded into $C(0,T,H).$

\begin{theorem}\label{continuity of the solution}
 Assume that $\fra:[0,T]\times V\times V\to \C$   satisfies  (\ref{eq:continuity-nonaut})-(\ref{eq 3:Dini-condition}) with $D(A(0)^{1/2})=V.$	Let $\Gamma$ be a another subdivision of $[0,T]$ that is finer than $\Lambda.$ Then
\begin{equation}\label{estimation of the error of the convergence2}
\|u_\Gamma-u_\Lambda\|_{C(0,T,V)}\leq \textbf{c}\Big(\omega(2|\Lambda|)+\frac{\omega(2|\Gamma|)}{|\Gamma|^{\gamma/2}}+\frac{\omega(2|\Lambda|)}{|\Lambda|^{\gamma/2}}+\int_0^{2|\Lambda|}\frac{\omega(s)}{s^{1+\gamma/2}}{\rm d}s\Big)\Big[\|u_0\|_V+\|f\|_{L^2(0,T,H)}\Big]
\end{equation}
for some positive constant $\textbf{c}>0$ depending only on $M,\alpha,\gamma$ and $c_H.$
\end{theorem}
\begin{proof}

We will use the notation of the the previous sections and we will proceed, as in the proof of Theorem
\ref{convergence uniform}, in several steps. To this end, we will adapt the proof of \cite[Theorem 4.4]{Ar-Mo15} to our situation.

\par\noindent \textit{Step a:}  By using  (\ref{Eq2: estimation resolvent}) and (\ref{Eq6: estimation resolvent}) in Proposition \ref{lemma: estimations for general from} for $(\lambda-\A_\Gamma(t))^{-1}$ and $(\lambda-\A_\Lambda(t))^{-1},$ respectively, and (\ref{Eq2:Dini condition operators}) we obtain  for every $t\in[0,T]$ that
\begin{align*}
\|u_{1,\Lambda}(t)&-u_{1,\Gamma}(t)\|_V\leq \frac{1}{2\pi}\int_\Gamma e^{-t\Re \lambda}\|(\lambda-\A_\Lambda(t))^{-1}(\A_\Lambda(t)-\A_\Gamma (t))(\lambda-\A_\Gamma(t))^{-1}u_0\|_V d\lambda
\\&\leq\frac{2c^2\omega(2|\Lambda|)}{\pi}\int_\Gamma \frac{e^{-t\Re \lambda}}{(1+|\lambda|)^\frac{3-\gamma}{2}}d\lambda \|u_0\|_V
\\&\leq \frac{2c^2\omega(2|\Lambda|)}{\pi}\Big(\int_0^\infty \frac{1}{(1+r)^\frac{3-\gamma}{2}}d r\Big) \|u_0\|_V.
\end{align*}

\par\noindent \textit{Step b:} Again  the estimates  (\ref{Eq4: estimation resolvent}) and  (\ref{Eq6: estimation resolvent}) in Proposition \ref{lemma: estimations for general from} and formula (\ref{Eq2:Dini condition operators}) imply that
\[\|(\lambda-\A_\Lambda(t))^{-1}(\A_\Lambda(t)-\A_\Gamma(t))(\lambda-\A_\Gamma(t))^{-1}f(s)\|_V\leq 2c^2\omega(2|\Lambda|) \frac{\|f(s)\|_H}{(1+|\lambda|)^{1-\frac{\gamma}{2}}}.\]
Therefore, we obtain by using  Fubini's theorem that
for all $\lambda\in \Gamma\setminus\{0\}$
\begin{align*}
\|u_{2,\Lambda}(t)&-u_{2,\Gamma}(t)\|_V=\|\int_0^t\frac{1}{2i\pi}\int_\Gamma e^{-(t-s) \lambda}(\lambda-\A_\Lambda(t))^{-1}(\A_\Lambda(t)-\A_\Gamma(t))(\lambda-\A_\Gamma(t))^{-1}f(s) d\lambda ds\|_V
\\&\leq \frac{c^2\omega(2|\Lambda|)}{\pi}\int_0^\infty \frac{1}{(1+r)^{1-\frac{\gamma}{2}}}\Big(\int_0^t e^{-(t-s)r\cos(\nu)}\|f(s)\|_H ds\Big) dr
\\&\leq \frac{c^2\omega(2|\Lambda|)}{\pi}\|f\|_{L^2(0,T,H)}\int_0^\infty \frac{1}{(1+r)^{1-\frac{\gamma}{2}}}\Big(\int_0^t e^{-2(t-s)r\cos(\nu)}ds\Big)^{1/2}  dr
\\&\leq \frac{c^2\omega(2|\Lambda|)}{\sqrt{2\cos(\nu)}}\Big(\int_0^\infty \frac{1}{\sqrt{r}(1+r)^{1-\frac{\gamma}{2}}}dr\Big)\|f\|_{L^2(0,T,H)}.
\end{align*}

\par\noindent\textit{Step c:} It remains to estimate $\|u_{3,\Lambda}(\cdot)-u_{3,\Gamma}(\cdot)\|_V.$ For this for each $h\in C(0,T,V)$ we set
\[(P_\Lambda h)(t):=\int_{0}^te^{-(t-s)A_\Lambda(t)}(\A_\Lambda(t)-\A_\Lambda(s))h(s){\rm  d}s.\]
From \cite[Lemma 4.5]{Ar-Mo15} we have $P_\Lambda h\in C(0,T,V).$ Thanks to Proposition  \ref{Prop: Dini condition for Linear-approximation}
and Lemma \ref{Lemma:Dini condition approximation operators} one can prove in a similar way as in Step 3 of the proof of \cite[Theorem 4.4]{Ar-Mo15} (see also Step 3 of the proof of Lemma \ref{Lemma: Invertibility of I-QLambda}) that $\|P_\Lambda\|_{\L(C(0,T,H)}\leq 1/2$ and thus $I-P_\Lambda$ is invertible on $\L(C(0,T,X)).$ Therefore, we obtain by using  the representation formula (\ref{Formula for the solution})
\begin{align} \label{Eq 1:Theorem: proof continuity}
u_{\Lambda,3}-u_{\Gamma,3}=(I-P_\Lambda)^{-1}&(u_{\Lambda,1}-u_{\Gamma,1})+(I-P_\Gamma)^{-1}(u_{\Lambda,2}-u_{\Gamma,2})\\&\qquad +
[(I-P_\Lambda)^{-1}(P_\Lambda-P_\Gamma)(I-P_\Gamma)^{-1}(u_{\Gamma,1}+u_{\Gamma,2})
\end{align}
\par\noindent The term on the right hand side of (\ref{Eq 1:Theorem: proof continuity}) is treated in \textit{Step} a)-b). We need only to  estimate the difference $P_\Lambda-P_\Gamma$ on $\L(C(0,T,V)).$  For each $h\in C(0,T,V)$ and $t\in [0,T]$ we have
\begin{align*}
(&P_\Lambda h -P_\Gamma h)(t)\\&=\int_{0}^te^{-(t-s)A_\Lambda(t)}\big[\A_\Lambda(t)-\A_\Lambda(s)-\A_\Gamma(t)+\A_\Gamma(s)\big]h(s){\rm  d}s
\\& \qquad\qquad +\int_{0}^t\big[e^{-(t-s)A_\Lambda(t)}-e^{-(t-s)A_\Gamma(t)}\big](\A_\Gamma(t)-\A_\Gamma(s))h(s){\rm  d}s.
\\&=\int_{0}^te^{-(t-s)A_\Lambda(t)}\big[\A_\Lambda(t)-\A_\Lambda(s)-\A_\Gamma(t)+\A_\Gamma(s)\big]h(s){\rm  d}s
\\&+\int_{0}^t\frac{1}{2i\pi}\int_\Gamma e^{-(t-s)\lambda}(\lambda-\A_\Lambda)^{-1}(A_\Lambda(t)-A_\Gamma(t))(\lambda-\A_\Gamma)^{-1}(\A_\Gamma(t)-\A_\Gamma(s))h(s){\rm  d}s.
\end{align*}
By  the estimate  (\ref{Eq6: estimation resolvent}) in Proposition \ref{lemma: estimations for general from} and the formula (\ref{Eq2:Dini condition operators}),
\begin{align*}
\|(\lambda-\A_\Lambda)^{-1}(A_\Lambda(t)-&A_\Gamma(t))(\lambda-\A_\Gamma)^{-1}(\A_\Gamma(t)-\A_\Gamma(s))h(s)\|_V
\\&\leq 4c^2\omega(2|\Lambda|)\frac{\omega_\Gamma(t-s)}{(1+|\lambda|)^{1-\gamma}}\|h(s)\|_V.
\end{align*}
Thus remarking that 
$\omega_\Gamma(t)\leq 2\omega(t)$ for every ${t\in[0,T]},$ it follows that
\begin{align*}
\|\int_{0}^t&\frac{1}{2i\pi}\int_\Gamma \big[e^{-(t-s)\lambda}(\lambda-\A_\Lambda)^{-1}(A_\Lambda(t)-A_\Gamma(t))(\lambda-\A_\Gamma)^{-1}\big](\A_\Gamma(t)-\A_\Gamma(s))h(s){\rm  d}s \|_V
\\&\leq 2c^2\omega(2|\Lambda|)\int_{0}^t\frac{1}{\pi}\int_0^\infty e^{-(t-s)r\cos(\theta)}\frac{\omega_\Lambda(t-s)}{r^{1-\gamma}}\|h(s)\|_V dr ds
\\&= 2c\omega(2|\Lambda|)\int_{0}^t\frac{1}{\pi}\int_0^\infty e^{-\rho\cos(\theta)}\frac{\omega_\Lambda(t-s)}{\rho^{1-\gamma}}(t-s)^{1-\gamma}\|h(s)\|_V \frac{d\rho}{t-s} ds
\\&\leq 2\omega(2|\Lambda|)\frac{c^2}{\pi}\int_{0}^t\Big(\int_0^\infty \frac{e^{-\rho\cos(\nu)}}{\rho^{1-\gamma}}d\rho\Big)\frac{\omega_\Gamma(t-s)}{(t-s)^{\gamma}}\|h(s)\|_V d\rho ds
\\& \leq {\bf c }\omega(2|\Lambda|) \|h\|_{C(0,T,V)}.
\end{align*}
for some constant ${\bf c}={\bf c}(\alpha, M,\gamma, T)$ independent of $\Gamma$ and $\Lambda.$ Next, writing
\begin{align*}e^{-(t-s)A_\Lambda(t)}&\big[\A_\Lambda(t)-\A_\Lambda(s)-\A_\Gamma(t)+\A_\Gamma(s)\big]h(s)
\\&=A^{-1/2}_\Lambda(t)A^{1/2}_\Lambda(t)e^{-(t-s)A_\Lambda(t)}\big[\A_\Lambda(t)-\A_\Lambda(s)-\A_\Gamma(t)+\A_\Gamma(s)\big]
\end{align*}
we obtain by (\ref{Eq3: estimation  semigroup}) and (\ref{Eq4: estimation  semigroup}) in Proposition \ref{lemma: estimations for general from} and since  $e^{-\cdot A_\Lambda(t)}$ is an analytic $C_0$-semigroup on $V$
\begin{align*}
&\int_{0}^t\|e^{-(t-s)A_\Lambda(t)}\big[\A_\Lambda(t)-\A_\Lambda(s)-\A_\Gamma(t)+\A_\Gamma(s)\big]h(s)\|_V{\rm  d}s
\\ &\leq  c^32^{1+\gamma/2}\int_{0}^t \kappa_{\Lambda,\Gamma}(s)\|h(t-s)\|_V{\rm  d}s
\end{align*}
where
\begin{equation*}\label{Kappa2}\kappa_{\Lambda}
(t):=\left\{%
\begin{array}{ll}
     [\omega_ \Gamma(t)+\omega_\Lambda(t)]t^{-(1+\frac{\gamma}{2})} & \hbox{ if } \  0\leq  t<2|\Lambda|,\\
    4\omega(2|\Lambda|)t^{-(1+\frac{\gamma}{2})}& \hbox{ if } \  2|\Lambda|< t\leq 2T, \\
   0& \hbox{ if } \   t\in ]-\infty,0]\cap]2T,+\infty[. \\
\end{array}%
\right. \end{equation*}
Because of  (\ref{eq 2:Dini-condition}) and (\ref{eq 2:Dini-condition for Linear-approximation}), the function
$t\mapsto k_{\Lambda,\Gamma}(t)$ belongs to $L^1(\R)$, and by a simple calculation we obtain

\[\|\kappa_\Lambda\|_{L^1(\R)}\leq \textbf{c}\Big(\frac{\omega(2|\Lambda|)}{|\Lambda|^{\gamma/2}}+\frac{\omega(2|\Gamma|)}{|\Gamma|^{\gamma/2}}+\int_0^{2|\Lambda|}\frac{\omega(s)}{s^{1+\gamma/2}}{\rm d}s\Big)\
\]
for some positive constant $\textbf{c}>0$ depending only on $M,\alpha, \gamma$ and $c_H.$ We conclude then
\[\|u_{3,\Lambda}(t)-u_{3,\Gamma}(t)\|_V\leq \textbf{c}\Big(\omega(2|\Lambda|)+\frac{\omega(2|\Lambda|)}{|\Lambda|^{\gamma/2}}+\frac{\omega(2|\Gamma|)}{|\Gamma|^{\gamma/2}}+\int_0^{2|\Lambda|}\frac{\omega(s)}{s^{1+\gamma/2}}{\rm d}s\Big)\Big[\|u_0\|_V+\|f\|_{L^2(0,T,H)}\Big].\]
for some positive constant $\textbf{c}>0$ (probably different from the previous one) depending only on $M,\alpha, \gamma$ and $c_H,$ and the proof is complete.

\end{proof}
We finish this section with our main result. 
\begin{theorem}\begin{enumerate}
\item  There exists a null sequence $(t_n)_{n\in\N}\subset [0,T]$ depending on $\omega$ such that for every $\varepsilon>0$ there exists $n_0\in\N$ such that for all $n\geq n_0$ one has
\[\|u_{\Gamma_n}-u_{\Lambda_n}\|_{C(0,T,V)}\leq \varepsilon \big[\|u_0\|_V+\|f\|_{L^2(0,T,H)}\big]\]
for all subdivision $\Lambda_n$ and $\Gamma_n$ of $[0,T]$ with $|\Gamma_n|<|\Lambda_n|=\frac{t_n}{2}.$
\item $u\in C([0,T],V)$ and there exists a null sequence $(t_n)_{n\in\N}\subset [0,T]$ depending on $\omega$ such that for every $\varepsilon>0$ there exists $n_0\in\N$ such that for all $n\geq n_0$ one has
\[\|u_-u_{\Lambda_n}\|_{C(0,T,V)}\leq \varepsilon \big[\|u_0\|_V+\|f\|_{L^2(0,T,H)}\big]\]
for all subdivision $\Lambda_n$ of $[0,T]$ with $|\Lambda_n|=\frac{t_n}{2}.$
\item Assume moreover that (\ref{Condition supplementary 2})  holds. Then $u_\Lambda\longrightarrow u$ in $C([0,T],V)$ uniformly on $u_0$ and $f$ as $|\Lambda|\to 0.$ More precisely,
\begin{equation*}
\|u-u_\Lambda\|_{C(0,T,V)}\leq \textbf{c}\Big(\omega(2|\Lambda|)+\frac{\omega(2|\Lambda|)}{|\Lambda|^{\gamma/2}}+\int_0^{2|\Lambda|}\frac{\omega(s)}{s^{1+\gamma/2}}{\rm d}s\Big)\Big[\|u_0\|_V+\|f\|_{L^2(0,T,H)}\Big]
\end{equation*}
for all subdivision $\Lambda$ and for some positive constant $\textbf{c}>0$ independent of $\Lambda.$
\end{enumerate}
\end{theorem}
\begin{proof} Assertion $1)$ follows by similar argument as in the proof of Corollary \ref{Corollary 1: uniform convergence}.
By Theorem \ref{continuity of the solution},  $u_\Lambda$ is a Cauchy sequence, and thus converges  in $C([0,T],V).$ In other hand, we known that $u_\Lambda \longrightarrow u$ strongly in $C([0,T],H).$ Therefore, assertion $2)$ is a direct consequence of $1)$ and the assertion $3)$ follows directly from the estimate (\ref{estimation of the error of the convergence2}) and the additional condition (\ref{Condition supplementary 2}).
\end{proof}

\end{document}